\documentclass[10pt,reqno,a4paper]{amsart}
\usepackage{graphicx}
\usepackage{amssymb}
\usepackage{amsfonts}
\usepackage{enumitem} % customize enumerates
\usepackage{hyperref}\hypersetup{colorlinks=true, citecolor=blue}
\usepackage{color,todonotes}
\usepackage[mathscr]{euscript}
\usepackage{comment} 
\usepackage{xfrac,esint,amsfonts,amsmath,amssymb,epsfig,mathrsfs,bm,accents,yfonts,mathtools,dsfont,stmaryrd}%MnSymbol}
\usepackage{graphicx,color}
\usepackage{tikz,fp,ifthen,fullpage}
\usepackage{pgfmath}
\usetikzlibrary{backgrounds}
\usetikzlibrary{decorations.pathmorphing,backgrounds,fit,calc,through}
\usetikzlibrary{arrows}
\usetikzlibrary{shapes,decorations,shadows}
\usetikzlibrary{fadings}
\usetikzlibrary{patterns}
\usetikzlibrary{mindmap}
\usetikzlibrary{decorations.text}
\usetikzlibrary{decorations.shapes}

\usepackage{pgfplotstable}
\usepackage{amsmath}
\usepackage{indentfirst}
\usepackage{amsmath,amssymb}
\usepackage{subfiles}
\usepgfplotslibrary{fillbetween}
\usetikzlibrary{patterns}
\usepackage{fancyhdr}
\usepackage{noindentafter}
\usepackage{mathrsfs}
\usepackage{environ}

\usepackage{tikz}
\usepackage{subfig}
\usepackage{caption}
\usepackage{enumerate}
\usepackage{tikz-cd}

\theoremstyle{plain} %to change the headings as suggested by Kleiner
\newtheorem{theorem}{Theorem}[section]
\newtheorem{lem}[theorem]{Lemma}
\newtheorem{prop}[theorem]{Proposition}
\newtheorem{thm}[theorem]{Theorem}
\newtheorem{cor}[theorem]{Corollary}

\newtheorem{remark}[theorem]{Remark}

\newtheorem{defn}[theorem]{Definition}
\numberwithin{equation}{section}

\newtheoremstyle{mytheorem}
{}% measure of space to leave above the theorem. E.g.: 3pt
{}% measure of space to leave below the theorem. E.g.: 3pt
{\it}% name of font to use in the body of the theorem
{\parindent}% measure of space to indent
{\bf}% name of head font
{.}% punctuation between head and body
{ }% space after theorem head; " " = normal interword space
{\thmnumber{#2.~}\thmname{#1}\thmnote{~\rm#3}}% Manually specify head

\newtheoremstyle{myremark}
{}% measure of space to leave above the theorem. E.g.: 3pt
{}% measure of space to leave below the theorem. E.g.: 3pt
{\rm}% name of font to use in the body of the theorem
{\parindent}% measure of space to indent
{\bf}% name of head font
{.}% punctuation between head and body
{ }% space after theorem head; " " = normal interword space
{\thmnumber{#2.~}\thmname{#1}\thmnote{~\rm#3}}% Manually specify head

\newtheoremstyle{myparagraph}
{}% measure of space to leave above the theorem. E.g.: 3pt
{}% measure of space to leave below the theorem. E.g.: 3pt
{\rm}% name of font to use in the body of the theorem
{\parindent}% measure of space to indent
{\bf}% name of head font
{.}% punctuation between head and body
{ }% space after theorem head; " " = normal interword space
{\thmnumber{#2.~}\thmname{#1}\thmnote{#3}}% Manually specify head

\makeatletter
\def\@secnumfont{\sc}
\def\section{\@startsection{section}{1}%
\z@{1.5\linespacing\@plus .2\linespacing}{.7\linespacing}%
{\normalfont\sc\centering}}
\makeatother
\makeatletter
\def\ps@headings{\ps@empty
 \def\@evenhead{%
  \setTrue{runhead}%
  \normalfont\footnotesize
  \rlap{\thepage}\hfil
  \def\thanks{\protect\thanks@warning}%
  \leftmark{}{}\hfil}%
 \def\@oddhead{%
  \setTrue{runhead}%
  \normalfont\footnotesize\hfil
  \def\thanks{\protect\thanks@warning}%
  \rightmark{}{}\hfil \llap{\thepage}}%
\let\@mkboth\markboth}
\makeatother
\makeatletter
\renewenvironment{proof}[1][\proofname]{\par
  \pushQED{\qed}%
  \normalfont \topsep6\p@\@plus6\p@\relax
  \trivlist
  \itemindent\normalparindent
  \item[\hskip\labelsep
    \bfseries
    #1\@addpunct{.}]\ignorespaces
}{%
  \popQED\endtrivlist\@endpefalse
}
\providecommand{\proofname}{Proof}
\makeatother
\newcommand{\Mass}{\mathds{M}}

\newcommand{\R}{\mathbb{R}}
\newcommand{\N}{\mathbb{N}}
\newcommand{\Z}{\mathbb{Z}}

\newcommand{\Haus}{\mathscr{H}}

\newcommand{\dV}{d_V\kern-1pt}
\newcommand{\de}{{\rm d}}

\newcommand{\cost}{{\mathcal C}}
\newcommand{\en}{\mathds E}

\DeclareMathOperator{\co}{co}
\DeclareMathOperator{\sign}{sign}

\DeclareMathOperator{\spt}{spt}
\newcommand{\trait}[3]{\vrule width #1ex height #2ex depth #3ex}
\newcommand{\trace}{\mathchoice%
  {\mathbin{\trait{.12}{1.2}{.03}\trait{.8}{0.09}{0.03}}}
  {\mathbin{\trait{.12}{1.2}{.03}\trait{.8}{0.09}{0.03}}}
  {\mathbin{\hskip.15ex\trait{.09}{.84}{0.02}\trait{.56}{.07}{.02}}\hskip.15ex}
  {\mathbin{\trait{.07}{.6}{.01}\trait{.4}{.06}{.01}}}}
\makeatletter
\@addtoreset{footnote}{section}
\makeatother
\title{A multi-material transport problem and its convex relaxation via rectifiable $G$-currents}

\author{Andrea Marchese}
\address{A. Mar.: Dipartimento di Matematica ``F. Casorati'', Universit\`a degli Studi di Pavia}
\email{andrea.marchese@unipv.it}

\author{Annalisa Massaccesi}
\address{A. Mas.: Department of Computer Science, University of Verona}
\email{annalisa.massaccesi@univr.it}

\author{Riccardo Tione}
\address{R.T.: Mathematik Institut der Universit\"at Z\"urich}
\email{riccardo.tione@math.uzh.ch}

\makeindex

\begin{document}

\maketitle

{
\footnotesize

	% abstract, keywords and MSC numbers
	%
{\sc Abstract.}
In this paper we study a variant of the branched transportation problem, that we call multi-material transport problem. This is a transportation problem, where distinct commodities are transported simultaneously along a network. The cost of the transportation depends on the network used to move the masses, as it is common in models studied in branched transportation. The main novelty is that in our model the cost per unit length of the network does not depend only on the total flow, but on the actual quantity of each commodity. This allows to take into account different interactions between the transported goods. We propose an Eulerian formulation of the discrete problem, describing the flow of each commodity through every point of the network. We prove existence of solutions under minimal assumptions on the cost. Moreover, we prove that, under mild additional assumptions, the problem can be rephrased as a mass minimization problem in a class of rectifiable currents with coefficients in a group, allowing to introduce a notion of calibration. The latter result is new even in the well studied framework of the ``single-material'' branched transportation.
\par
\medskip\noindent
{\sc MSC (2010): 49Q10, 49Q15, 49Q20, 53C38, 90B06, 90B10.
\par
}
}
%\tableofcontents

\section*{Introduction}

In this paper we study the \emph{multi-material transport problem} (MMTP). Informally, given two arrays
$$\mu^-=(\mu^-_1,\ldots,\mu^-_m),\quad\mu^+=(\mu^+_1,\ldots,\mu^+_m)$$ 
of discrete positive measures on $\R^d$, we study transportation networks between $\mu^-$ and $\mu^+$ of the form ${T=(T_1,\ldots,T_m)}$, where each $T_i$ is a vector valued measure on $\R^d$ (with values also in $\R^d$) having distributional divergence 
\begin{equation}\label{e:constraint}
{\rm div}(T_i)=\mu^-_i-\mu^+_i\qquad \forall\,i=1,\ldots,m\,.
\end{equation}
More precisely, we will consider only transportation networks with a certain structure, namely we require that there exists a 1-rectifiable set $E\subset\R^d$ (endowed with a unit vectorfield $\tau$, orienting the approximate tangent of $E$) and for every $i=1,\dots,m$ a subset $E_i\subset E$ and a multiplicity $\theta_i\in\Z$ such that 
$$T_i:=\theta_i\tau\Haus^1\trace E_i,$$
where the latter means that, for every continuous and compactly supported vectorfield $v$ on $\R^d$, we have
$$\langle T_i,v\rangle=\int_{E_i}\langle\tau(x); v(x)\rangle\theta_i(x)\, d\Haus^1(x)\,.$$
%Up to changing sign to the multiplicities $\theta_i$ we can assume that 
%{\color{blue} Possibly changing sign to the multiplicities $\theta_i$, we can assume that }there exists a unit vector field $\tau$ defined on $E:=\bigcup_{i=1}^m E_i$ such that 
%$$\tau(x)=\tau_i(x) \mbox{ for $\Haus^1$-a.e. $x\in E_i$, for every $i$}.$$ 
Then we associate to $T$: 
\begin{itemize}
\item[$\circ$] the multiplicity $\theta=(\theta_1,\dots,\theta_m)$, which is a function on $E$ with values in $\Z^m$, 
\item[$\circ$] the vector valued measure on $\R^d$ (with values in $\R^{d\times m}$)
$$T:=(\tau\otimes\theta)\Haus^1\trace E,$$ 
\item[$\circ$] and the energy
\begin{equation}\label{e_energ}
\en(T):=\int_E \cost(\theta)\,d\Haus^{1},
\end{equation}
where $\cost:\Z^m\to[0,+\infty)$ is a \emph{cost function}. 
\end{itemize}
{The MMTP consists in the minimization of the energy \eqref{e_energ} under the constraint \eqref{e:constraint}.}\\

We briefly step back for a heuristic introduction to the problem. In optimal transport problems one can focus on specific (concave) costs, which favor the aggregation of moved particles and generate optimal structures with branching. The branched transport problem is named after this peculiar phenomenon. A great interest has been devoted to branched transportation problems in the last years, providing several results concerning existence of solutions  \cite{Xia,Mad_So_Mo,Bernot2009,Ber_Ca_Mo,Bra_But_San,Pe,Co_DeRo_Mar_Stu}, regularity and stability \cite{Xia2,Bran_But_San,De_So,De_So2,Mo_San,Xia3,Bran_So,Bran_Wirth,Co_DeRo_Mar,Co_DeRo_Mar2,Co_DeRo_Mar3} and strategies to compute minimizers or to prove minimality of concrete configurations \cite{Ou_San,Bon_Le_San,Cham_Mer_Fer,Mas_Ou_Ve,Bo_Or_Ou,Bonafini,AnMa,AnMa2,Gold}. 

Nonetheless, to our knowledge only problems involving the transport of one (homogeneous) material have been studied and modeled as variational problems. These models do not apply in planning a network for the transportation of different goods, whose mutual interactions require a formulation which involves several variables. The easiest examples of natural multi-material transport problem concern mixed-use roads (where vehicles of different size and pedestrians are allowed to circulate) and the transport through vehicles for goods and passengers. 

Another notable example is given by the power line communications technology (PLC, see \cite{Fer_La_New_Swart, Do}), which uses the electric power distribution network for data transmission. PLC has been introduced in the United States of America more than a century ago and used for the communications on moving trains or, more generally, for maintenance operations of the electric power network. Recently, a special type of PLC, the broadband over power lines (BPL) is being studied and improved for high-speed data transmission, being particularly convenient for isolated areas. Electric power and data signals are impossible to be treated as a homogeneous ``material'' for several reasons the main one being the fact that the electricity and the Internet supply are subject to different costs, depending on the users' concentration and demands.

Similar problems, usually grouped under the name of multi-commodity flow problems, were studied (see e.g. \cite{Ford_Ful,E_I_Sha}) as minimization problems on graphs, often considering also constraints on the capacity of the network. Up to now, the aim of the research in this area was mainly devoted to study the complexity of the problem and to improve the efficiency of algorithms for numerical solutions.

The main results of the present paper are the existence of solutions to the minimization of the energy \eqref{e_energ} under the constraint \eqref{e:constraint}, with minimal assumptions on the cost $\cost$ (Theorem \ref{t:exist}) and, under mild additional assumptions, the equivalence between the MMTP and a mass minimization problem in a class of rectifiable currents with coefficients in a group (Theorem \ref{main}). The equivalence between the two problems allows us to introduce the notion of calibrations in this context. This was initiated in previous works (\cite{AnMa, AnMa2}) for ``single-material'' branched transportation problems and for very special choices of the cost functionals (i.e., the Steiner cost and the Gilbert-Steiner $\alpha$-mass, respectively) and the benefit of introducing calibrations in such contexts is witnessed e.g. by \cite{Mas_Ou_Ve, Bo_Or_Ou, Ca_Plu, Ca_Plu2}. Under our general assumptions on the cost, the equivalence result is new even in the case of ``single-material'' transportation problems. 
   
\section*{Acknowledgments}

The first named author is supported by the ERC-grant 306247 {\it Regularity of area minimizing
currents}. This project has received funding from the European Union's Horizon 2020 research and innovation programme under the grant agreement No 752018 (CuMiN). %vedete voi se inserire qui l'ERC di Camillo di superfici minime (se è l'ERC quello con le superfici minime...)
We warmly thank D. Vittone for valuable discussions and the anonymous referees for suggesting several improvements.
%We thank D. L. Vuitton for helping us improving the elegance of our proofs.

%%%%%%%%%%%%%%%%%%%%%%%%%%%%%%%%%%%%%%%%%%%%%%%%%%%%%%%%%%%%%%%%%%%%%%%%%%%%%%%%%%%%%%%%%%%%%%%%%%%%%%%%%%%%%%%%%%%%%%%%%%%%%%%

\section{Notation and preliminaries}
Consider a norm $\|\cdot\|$ on $\R^n$ and its dual norm $\|\cdot\|^*$. The Euclidean norm is instead denoted by $|\cdot|$ and we will always denote an orthonormal basis of $(\R^n, |\cdot|)$ by $\{{\bf e}_1,\ldots,{\bf e}_n\}$. The scalar product between vectors $v$ and $w$ of $\R^n$ is denoted $\langle v;w\rangle$. Our aim is to define ($1$-dimensional) {\em currents in $\R^d$ with coefficients in $(\R^n,\|\cdot\|)$}.  Through the paper we will always use $d$ for the dimension of the ambient space. In Section \ref{sec:2}, two quantities $m$ and $N$ will arise in our variational problem and we will work with currents with coefficients in $\R^m$ and $\R^N$, respectively. The letter $n$, used in this preliminary section, stands for the dimension of the vector space of coefficients of our currents. In the sequel one can refer to the definitions of this section replacing $n=m$ or $n=N$.\\ 

Currents with coefficients in a normed (abelian) group $(G,|\cdot|_G)$ have been introduced in \cite{Fle} and already studied by several authors (see \cite{White2,White1,De_Pauw_Hardt}). Our interest is restricted to the case $(G,|\cdot|_G)=(\R^n,\|\cdot\|)$, and we follow a ``non-standard'' approach, defining currents by duality with $\R^n$-valued differential forms in $\R^d$ (instead of completing the space of polyhedral $\R^n$-chains). With this approach we obtain an integral representation of currents, (see \eqref{e:int_repr}) which allows to introduce calibrations in a natural way.\\

We introduce now some notation about currents with coefficients in $(\R^n,\|\cdot\|)$. For the rest of this section, we will often drop the norm $\|\cdot\|$, meaning that we will write $\R^n$-valued 1-covector/differential form or 1-current with coefficients in $\R^n$ instead of $(\R^n,\|\cdot\|)$-valued 1-covector/differential 1-form or 1-currents with coefficients in $(\R^n,\|\cdot\|)$. We limit ourselves to define what is strictly necessary for the purposes of our paper. To begin with, we give the following definitions.
\begin{defn}[$\R^n$-valued 1-covector]{\rm
A map $\alpha: \R^d \times \R^n \to \R$
is an {\em $\R^n$-valued $1$-covector} in $\mathbb{R}^d$ if 
\begin{itemize}
\item[(i)] $\forall\,\tau \in \mathbb{R}^d,\ \alpha(\tau,\cdot) \in (\R^n)^*$;
\item[(ii)] $\forall\,\theta \in \R^n,\ \alpha(\cdot,\theta): \mathbb{R}^d \to \mathbb{R}$ is a ``classical'' $1$-covector.
\end{itemize}
The evaluation of $\alpha$ on the pair $(\tau,\theta)$ is also denoted by $\langle\alpha;\tau,\theta\rangle$. The space of $\R^n$-valued $1$-covectors in $\R^d$ is denoted by $\Lambda^1(\mathbb{R}^d;\R^n)$.
}\end{defn}

Observe that the space $\Lambda^1(\mathbb{R}^d;\R^n)$ is a normed vector space when endowed with the comass norm 
$$
\|\alpha\|_c:=\sup\{\|\alpha(\tau,\cdot) \|^*:\,|\tau| \leq 1, \tau \in \R^d\}.
$$
We can write the action of an $\R^n$-valued $1$-covector $\alpha$ as 
$$
\alpha(\tau,\theta)=\sum_{j = 1}^n\alpha_j(\tau)\langle{\bf e}_j;\theta\rangle,
$$ 
where, for $j=1,\ldots,n$, $\alpha_j:=\alpha(\cdot,{\bf e}_j)$ are the {\em components} of $\alpha$. 

Fix now a convex open set $U \subset \mathbb{R}^d$. It is clear that such assumption is not restrictive for most of the reasonable cases, nonetheless we remark that other choices of $U$ could change the homology class in which we set the variational problem.
\begin{defn}[$\R^n$-valued differential 1-form]{\rm 
An \emph{$\R^n$-valued differential 1-form in $U$} is a map $\omega:U \to \Lambda^1(\mathbb{R}^d;\R^n).$ We say that $\omega$ is \emph{smooth} if and only if every component $\omega_j$ belongs to $C^\infty(U;\Lambda^1(\R^d;\R))$, where the components of an $\R^n$-valued differential $1$-form are defined similarly to the components of an $\R^n$-valued $1$-covectors. We denote by $C_c^\infty(U;\Lambda^1(\mathbb{R}^d;\R^n))$ the vector space of smooth, $\R^n$-valued differential 1-forms, with compact support in $U$. 
}\end{defn}

Finally, we define the {\em comass norm} of the $\R^n$-valued differential $1$-form $\omega$ as
\[
\|\omega\|_c :=\sup_{x \in U} \|\omega(x)\|_c.
\]
The exterior derivative of an $\R^n$-valued function (0-form) is once again defined using the components.

\begin{defn}[Exterior derivative of an $\R^n$-valued 0-form]\label{2cur}{\rm
Let $\eta\in C_c^\infty(U;\R^n)$ be an $\R^n$-valued 0-form, and, for $j = 1,\dots,n$, denote with $\eta_j$ its components (i.e., $\eta_j:=\langle \eta;{\bf e}_j\rangle$). Then the {\em exterior derivative} of $\eta$ is the $\R^n$-valued differential 1-form which is defined componentwise by
\[
(\de\eta)_j := \de(\eta_j),\quad j=1,\ldots,n. 
\]
}\end{defn}

We are now ready to define 1-currents with coefficients in $\R^n$.
\begin{defn}[1-currents with coefficients in $\R^n$]{\rm
Let $T$ be a linear functional on $C_c^\infty(U;\Lambda^1(\mathbb{R}^d,\R^n))$. By definition, $T$ is continuous if $T(\omega^i)\to 0$ for every sequence of $\R^n$-valued differential 1-forms $\omega^i \in C^\infty(U;\Lambda^1(\mathbb{R}^d,\R^n))$ such that
\begin{itemize}
\item[(i)] $\spt(\omega^i) \subset K$, for some compact set $K\subset U$;
\item[(ii)] every component of $\omega^i$ converges uniformly to $0$ with all its derivatives when $i\to\infty$. 
\end{itemize}
The space of linear, continuous functionals on $C_c^\infty(U;\Lambda^1(\mathbb{R}^d,\R^n))$ is the space of \emph{$1$-currents in $U$ with coefficients in $\R^n$}%; it is denoted by $\mathcal{D}^{\R^n}_1(U)$
. We write $T^{i}\stackrel{*}{\rightharpoonup} T$ when the sequence of currents $(T^i)_{i\ge 1}$ with coefficients in $\R^n$ is weakly*-converging to $T$, i.e. when
\[
T^i(\omega) \to T(\omega)\quad\forall\,\omega \in C_c^\infty(U;\Lambda^1(\mathbb{R}^d,\R^n)).
\]
Furthermore, if $T$ is a 1-current with coefficients in $\R^n$, we define its {\em mass} as 
\[
{\Mass}(T):=\sup\{|T(\omega)| :\,\|\omega\|_c \leq 1\}.
\]

The {\em boundary} of $T$ is the $\R^n$-valued distribution $\partial T$ which fulfills the relation
\[
\partial T(\eta):=T(\de\eta)\quad\forall \eta \in C_c^\infty(U;\R^n).
\] 
Finally, when we mention the \emph{components} of the current $T$, we refer to the ``classical'' currents $T_j$, $j=1,\ldots,n$, defined as
\[
T_j(\omega):=T(\omega{\bf e}_j)\quad\forall\,\omega\in C^\infty(U;\Lambda^1(\R^d;\R)),
\]
where we denoted by $\omega{\bf e}_j$ the $\R^n$-valued differential 1-form whose $j$-th component coincides with $\omega$ and all other components are all null.
}\end{defn}

\begin{remark}{\rm
Analogously, the definitions of $\R^n$-valued $k$-covectors, $\R^n$-valued differential $k$-forms, and $k$-currents with coefficients in $\R^n$ are given by specifying their components, i.e. an array made of $n$ ``classical'' $k$-covectors, differential $k$-forms, and $k$-currents, respectively. Similarly, the definitions of the exterior derivative of an $\R^n$-valued differential $k$-form and of the boundary of a $k$-current with coefficients in $\R^n$ are understood.
}
\end{remark}
{\begin{remark}{\rm
Notice that, if $T$ is a current with coefficients in $\R^m$ with at most one non-trivial component $T_j$, then the mass $\Mass(T)$ differs from the classical mass of $T_j$ by a multiplicative constant: namely, the ratio between the Euclidean norm on $\R$ and the restriction of the norm $\|\cdot\|$ on span$(e_j)$. Therefore, to avoid a possibly misleading abuse of notation, we denote $\mathbb{M}$ the mass of classical currents. 
}
\end{remark}
}

%\begin{defn}
%The space of $1$-currents with finite mass in $U$ is
%\[
%\mathcal{M}^E_1(U) := \{T \in \mathcal{D}_1^E(U) : M(T) < +\infty\}.
%\] 
%and the space of normal currents is:
%\[
%\mathcal{N}^E_1(U) := \{T \in \mathcal{D}_1^E(U) : N(T) < +\infty\}.
%\] 
%\end{defn}
%We summarize in the following definition the special subsets of currents that we are going to consider.

We are going to consider the following special class of currents. We recall that a \emph{$1$-rectifiable set} $E\subset U$ is an $\Haus^1$-measurable set which can be covered, up to an $\Haus^1$-null subset, with the images of countably many curves of class $C^1$. A 1-rectifiable set $E$ has a well defined notion of tangent line at $\Haus^1$-a.e. point $x\in E$, which is denoted ${\rm  Tan}(E,x)$.
\begin{defn}[Rectifiable 1-currents with coefficients in $\Z^n$]{\rm 
%\begin{itemize}
%\item[(1)] A {\em normal current with coefficients in $\R^n$} is a current with coefficients in $\R^n$, with finite mass, whose boundary has also finite mass.
%\item[(2)] A {\em rectifiable current (with coefficients in $\Z^n$)} is a current with finite mass admitting a representation as
%\begin{equation}\label{e:int_repr}
%T(\omega)=\int_\Sigma \langle\omega(x);\xi(x),\theta(x)\rangle\,d\Haus^1(x)\,,
%\end{equation}
%where $\Sigma\subset U$ is an $\Haus^1$-measurable and countably $1$-rectifiable set, the orientation $\xi:\Sigma\to{\mathbb S}^{d-1}$ belongs to the approximated tangent space of $\Sigma$ ($\Haus^1$-almost everywhere) and the multiplicity $\theta:\Sigma\to \Z^n$ is summable (i.e., $\theta\in L^1_{\rm loc}(\Sigma;\Z^n)$). In this case, we say that the current $T$ is of type $\llbracket\Sigma,\xi,\theta\rrbracket$.
%\item[(3)] A {\em polyhedral current (with coefficients in $\Z^n$)} is a rectifiable current (with coefficients in $\Z^n$) with $\Sigma=\bigcup_{i\in I}\sigma_i$, where, for every $i\in I$, $\sigma_i$ is a convex $1$-dimensional cell in $\R^d$ and $\theta$ is constant on $\sigma_i$.
%\end{itemize}

A {\em rectifiable $1$-current in $U$, with coefficients in $\Z^n$} is a $1$-current with coefficients in $\R^n$ with finite mass admitting the integral representation
\begin{equation}\label{e:int_repr}
T(\omega)=\int_\Sigma \langle\omega(x);\xi(x),\theta(x)\rangle\,d\Haus^1(x)\quad\forall\,\omega\in C^\infty(U;\Lambda^1(\R^d;\R^n)),
\end{equation}
where $\Sigma\subset U$ is a countably $1$-rectifiable set, $\xi(x):\Sigma\to{\mathbb S}^{d-1}\cap {\rm{Tan}}(E,x)$ for $\Haus^1$-a.e. $x$ is called the orientation, and $\theta\in L^1_{\rm loc}(\Sigma;\Z^n)$ is the multiplicity. We denote such current $T$ as $\llbracket\Sigma,\xi,\theta\rrbracket$.
}\end{defn} 

We have the following characterization of the mass of a rectifiable current (see \cite[26.8]{Si} for the analogous statement for ``classical'' currents).
\begin{lem}[Characterization of the mass]
If $T=\llbracket\Sigma,\xi,\theta\rrbracket$ is a rectifiable 1-current with coefficients in $\Z^n$, then
\[
{\Mass}(T) = \int_\Sigma \|\theta(x)\|\,d\Haus^1(x).
\]
\end{lem}

For the rest of the paper, we mainly focus on rectifiable 1-currents with coefficients in $\Z^n$ whose boundary has finite mass. With a small abuse of notation we call them {\em 1-dimensional integral $\Z^n$-currents}. The following structure theorem for 1-dimensional integral $\Z^n$-currents is an immediate consequence of its counterpart for ``classical'' integral 1-currents (\cite[4.2.25]{Fe1}). See also \cite[Theorem 2.5]{Con_Gar_Mas}). Roughly speaking, it states that a 1-dimensional integral $\Z^n$-current can be thought simply as a superposition of oriented curves with (vectorial) multiplicities.

\begin{thm}[Structure of 1-dimensional integral $\Z^n$-currents]\label{thm:stru}
Let $T=\llbracket\Sigma,\tau,\theta\rrbracket$ be a 1-dimensional integral $\Z^n$-current in $U$. Then
\[
T=\sum_{k=1}^{M} \tilde T^k+\sum_{h=1}^\infty \mathring{T}^{h},
\]
where:
\begin{itemize}
\item[$\circ$] $\tilde T^k=\llbracket\Gamma_k,\tau_k,\tilde\theta_k\rrbracket$, being $\Gamma_k$ the image of an injective, Lipschitz, open curve $\gamma_k:[0,1]\to U$, $\tau_k(\gamma_k(t))=\frac{\overset{\cdot}\gamma_k(t)}{|\overset{\cdot}\gamma_k(t)|}$ for a.e. $t$, and $\tilde\theta_k\in\Z^n$ being constant on $\Gamma_k$. Moreover, for $j=1,\dots,n$ it holds 
\begin{equation}\label{boundmultipl}
\sum_{k=1}^M|(\tilde\theta_k(x))_j|\leq\frac{1}{2}\mathbb{M}(\partial T_j),\quad\mbox{for $\Haus^1$ a.e. $x\in\bigcup_{k=1}^M\Gamma_k$}.
\end{equation}
Additionally, for $j=1,\dots,n$ it holds
\begin{equation}\label{boundmultipl2}
\left|\sum_{k=1}^M\langle\tau_k(x);\tau(x)\rangle(\tilde\theta_k(x))_j\right|\leq |(\theta(x))_j|,\quad\mbox{for $\Haus^1$ a.e. $x\in\bigcup_{k=1}^M\Gamma_k$}
\end{equation}
and in the inequality above the two quantities $\sum_{k=1}^M\langle\tau_k(x);\tau(x)\rangle(\tilde\theta_k(x))_j$ and $(\theta(x))_j$ have the same sign.
\item[$\circ$] $\mathring{T}^h=\llbracket Z_h,\nu_h,\mathring\theta_h\rrbracket$, being $Z_h$ the image of a Lipschitz, closed curve $\zeta_h:[0,1]\to U$, which is injective on $(0,1)$, $\nu_h(\zeta_h(t))=\frac{\overset{\cdot}\zeta_h(t)}{|\overset{\cdot}\zeta_h(t)|}$ for a.e. $t$, and being $\mathring\theta_h\in\Z^n$ constant on $Z_h$. 
\end{itemize}
\end{thm}

\subsection{Compactness}
The following compactness theorem holds:

\begin{theorem}[Compactness]\label{thm:cptness}
Consider a sequence $(T^l)_{l\in\N}$ of 1-dimensional integral $\Z^n$-currents in $U$ such that
\[
\sup_{l\in\N}\left(\Mass(T^l)+\Mass(\partial T^l)\right)<+\infty\,.
\]
Then there exists a 1-dimensional integral $\Z^n$-current $T$ in $U$ and a subsequence $\left(T^{l_r}\right)_{r\in\N}$ such that
\[
T^{l_r}\stackrel{*}{\rightharpoonup} T\,.
\]
Moreover it holds
$$\liminf_{r\to\infty}\Mass(T^{l_r})\geq \Mass(T).$$
\end{theorem}
The proof of this result is a straightforward application of the Closure Theorem for integral currents (see \cite[4.2.16]{Fe1}) to each component $T^l_j$ of the elements of the sequence $(T^l)_{l\in\N}$. The lower semicontinuity of the mass is straightforward. By direct methods we get the existence of a mass-minimizing rectifiable current for a given boundary.

\begin{cor}\label{cor:dirmeth}
Let $T^\flat$ be a 1-dimensional integral $\Z^n$-current in $U$. Then there exists a 1-dimensional integral $\Z^n$-current $T^\sharp$ in $U$ such that
\[
\Mass(T^\sharp)=\min_{\partial T=\partial T^\flat}\Mass(T)\,,
\]
where the minimum is computed among 1-dimensional integral $\Z^n$-currents in $U$.
\end{cor}
%cone construction
% \begin{remark}\label{rmk:cone}{\rm
% Consider the sets $X,Y\subset\R^d$ with $X=\{x_1,\ldots,x_n\}$, $Y=\{y_1,\ldots,y_n\}$ and\footnote{Notice that $x_1,\ldots,x_n$ are not necessarily different in $X$ and the same holds for $Y=\{y_1,\ldots,y_n\}$.} trivial intersection $X\cap Y=\emptyset$. We can always define a rectifiable current with coefficients in $\Z^d$ as
% \[
% \hat Z:=\sum_{j=1}^n \llbracket \sigma_j,\tau_j,g_j\rrbracket\,,
% \]
% where $\sigma_j:[0,|x_j-y_j|]\to\R^d$ is the segment between $x_j$ and $y_j$ and $\tau_j\in\Lambda_1(\R^d)$ is the tangent \((y_j-x_j)/|y_j-x_j|\). It turns out that
% \[
% \partial\hat Z=B^{(n,\alpha)}_{X,Y}:=\sum_{j=1}^ng_j(\delta_{y_j}-\delta_{x_j})\,.
% \]
% Therefore, thanks to Corollary \ref{cor:dirmeth}, given two finite sets $X,Y\subset\R^d$ we can always solve the mass minimization problem for the boundary $B^{(n,\alpha)}_{X,Y}$.}
% \end{remark}

%calibrations
\subsection{Calibrations}
The main advantage of proving the equivalence between the MMTP and a mass minimization problem is that, in the latter case, we can make use of calibrations to prove minimality.
\begin{defn}[Calibration]\label{def:calib}{\rm
Consider a rectifiable $1$-current $T=\llbracket\Sigma,\tau,\theta\rrbracket$ in $U$, with coefficients in $\Z^n$. A smooth $\R^n$-valued differential $1$-form $\omega$ in $U$ is a \emph{calibration} for $T$ if the following conditions hold:
\begin{itemize}
\item[(i)] for a.e. $x\in\Sigma$ we have that $\langle\omega(x);\tau(x),\theta(x)\rangle=\|\theta(x)\|$;
\item[(ii)] the form is \emph{closed}, i.e., ${\rm d}\omega=0$;
\item[(iii)] for every $x\in U$, every unit vector $\tau\in\R^d$ and every $\eta\in\R^n$ we have that
\[
\langle\omega(x);\tau,\eta\rangle\le \|\eta\|\,.
\]
\end{itemize}
}\end{defn}
The existence of a calibration is a sufficient condition for minimality.
\begin{theorem}[Minimality of calibrated currents]\label{thm_calib}
Let $T=\llbracket\Sigma,\tau,\theta\rrbracket$ be a rectifiable $1$-current in $U$, with coefficients in $\Z^n$, and let $\omega$ be a calibration for $T$. Then $T$ minimizes the mass among rectifiable 1-currents in $U$, with coefficients in $\Z^n$, with the same boundary $\partial T$.
\end{theorem}
\begin{proof}
A competitor $T'=\llbracket\Sigma',\tau',\theta'\rrbracket$ satisfies $\partial T'=\partial T$. Since $U$ is convex, there exists a $2$-dimensional current $R$ in $U$, with coefficients in $\R^n$, such that $\partial R=T-T'$.  As a consequence, together with the properties of $\omega$ listed in Definition \ref{def:calib}, we obtain that
\begin{align*}
\Mass(T) & = \int_\Sigma \|\theta(x)\|\,d\Haus^{1}(x)\\
&\stackrel{{\rm (i)}}{=}\int_\Sigma\langle\omega(x);\tau(x),\theta(x)\rangle\,d\Haus^{1}(x)= \partial R(\omega)+T'(\omega)\\
&\stackrel{{\rm (ii)}}{=}\int_{\Sigma'}\langle\omega(x);\tau'(x),\theta'(x)\rangle\,d\Haus^{1}(x)\\
&\stackrel{{\rm (iii)}}{\le}\int_{\Sigma'}\|\theta'(x)\|\,d\Haus^{1}(x)=\Mass(T')\,.\qedhere
\end{align*}
\end{proof}
\begin{comment}
\begin{remark}\label{calib_rect}
 In \S \ref{sec:tree}, we need to allow also the possibility that the competitor $Z'$ is a rectifiable current with coefficients in $(\R^n,\|\cdot\|_\alpha)$. The
existence of a calibration
for $Z$ guarantees the optimality of the mass of $Z$ even in this larger class. The proof of this fact is the same as that given above.
\end{remark}
\end{comment}

\section{Multi-material transport problem}\label{sec:2}\label{nota}%qui problema discreto
In this section, we define the multi-material transport problem and we state the main result of the paper. First of all, let us introduce some notation. Our ambient is the Euclidean space $\mathbb{R}^d$. 
%and a vector $(N_1,\dots,N_m)\in\N^m$. Let us put $N :=\sum_{j=1}^m N_j$. 
For $n=1,2,\dots$, we consider the following partial order on $\R^n$, where the coordinates are always expressed with respect to the standard basis $\{{\bf e}_1,\ldots,{\bf e}_n\}$. Given two vectors $x=(x_1,\dots,x_n)$ and $y=(y_1,\dots,y_n)$, we write $x\preceq y$ if and only if 
\begin{equation}\label{order}
\mbox{$|x_j| \leq |y_j|$\quad  and \quad $x_jy_j \geq 0$, \quad $\forall j=1,\dots,n$}.
\end{equation}
We say that a norm $\|\cdot\|$ on $\mathbb{R}^n$ is \emph{monotone} if $\|x\| \leq \|y\|$, for every $x \preceq y \in \mathbb{R}^n$. We say that $\|\cdot\|$ is \emph{absolute} if $\|x\| = \|\bar x\|, \forall x\in \mathbb{R}^n$, where $\bar x = (|x_1|,\dots,|x_n|)$. Now we fix an integer $m\in\N$ (which represents the number of different types of commodities involved in the transportation problem).
\begin{defn}[Multi-material cost]\label{def:multi-material_cost}
A \emph{multi-material cost} is a function $\cost:\Z^m\to[0,+\infty)$ with the following properties:
\begin{itemize}
\item[{\rm (i)}] $\cost$ is even, i.e. $\cost(x)=\cost(-x)$, and $\cost(x)=0$ if and only if $x=0$;
\item[{\rm (ii)}] $\cost$ is increasing, i.e., $\cost(x)\le\cost(y)$ for every $x\preceq y$;
\item[{\rm (iii)}] $\cost$ is subadditive, i.e. $\cost(x+y)\leq\cost(x)+\cost(y)$ for every $x,y\in \Z^m$.
\end{itemize}
In order to prove the equivalence between the MMTP and a mass minimization problem, we will replace {\rm(iii)} with a stronger property, namely
\begin{itemize}
\item[${\rm (iii')}$] there exists a monotone norm $\|\cdot\|_{\star}$ on $\R^m$ with respect to which $\cost$ is sublinear, i.e. $\frac{\cost(x)}{\|x\|_\star}\le\frac{\cost(x')}{\|x'\|_{\star}}$ for every $x,x'\in \Z^m\setminus\{0\}$ with $x'\preceq x$.
\end{itemize} 
\end{defn}

\begin{remark}[Extension of multi-material costs]
{\rm If $\cost$ is defined only on a rectangle $$R:=[-a_1,a_1]\times\dots,\times[-a_m,a_m]\subset \Z^m$$ and it satisfies {\rm(i), (ii), (iii)} (respectively {\rm(i), (ii), ${\rm (iii')}$}) on $R$, then one can define a new cost 
$\bar\cost:\Z^m\to[0,+\infty]$ defining
$$\bar\cost(x):=\max_{y\in R}\{\cost(y):y\preceq x\}.$$
One can see immediately that the cost $\bar\cost$ satisfies {\rm(i), (ii), (iii)} (respectively {\rm(i), (ii), ${\rm (iii')}$).}}
\end{remark}

%
%We use the following notation for multi-indices.
%\begin{itemize}
%\item $I = (i_1,\dots,i_p)$ multi-indices with integer entries, i.e.:
%\[
%i_j \in \mathbb{N}, i_j \leq i_{j + 1}, i_p \leq N, \forall j \in \{1,\dots,p\}. 
%\]
%Moreover, let $\mathcal{A}$ be the set of multi-indices with at most $N$ elements. $|I|$ is the length of $I$, that is $p$ in the previous case. $I_i$ is the multi-index obtained from $I$ considering only the elements $i_l$ between $\sum_{k = 1}^iN_{k - 1} + 1$ and $\sum_{k = 1}^iN_k$, with $N_0 := 0$. In particular, $I = (I_1,\dots, I_m)$. Moreover, if $J,I \in \mathcal{A}$, $J = (j_1,\dots,j_k)$, $I = (i_1,\dots,i_p)$, $k\leq p$ and $j_l = i_{k_l}$ for every $l \in \{1,\dots,k\}$ and for some $k_l \in \{1,\dots, p\}$, then we will write $J\subseteq I$;
%\item $x\leq y$ means the following (partial) ordering relation: consider two vectors $x$ and $y$ in $\mathbb{R}^d$. We write $x\leq y$ if and only if $|x_i| \leq |y_i|$ and $x_iy_i \geq 0$, $\forall i$;
%\item $e_I$ a vector of the form $e_{i_1} + \dots + e_{i_p}$, where $I \in \mathcal{A}$, $I = (i_1,\dots,i_p)$. Observe that $J\subseteq I$ if and only if $e_J \leq e_I$.
%\end{itemize}

A multi-material cost induces a functional on 1-dimensional integral $\Z^m$-currents, that we denote $\en$. Given a 1-dimensional integral $\Z^m$-current $T=\llbracket \Sigma,\tau,\theta\rrbracket$, we denote its \rm{energy} by 

\begin{equation}\label{EN}
\en(T):=\int_\Sigma \cost(\theta)\,d\Haus^{1}\,.
\end{equation}

%If the boundary $B$ is concentrated on the points $\{p_1,\ldots,p_D\}\subset\R^d$, then it is represented as
%\[
%B=\sum_{\ell=1}^D \eta_\ell\delta_{p_\ell},\quad \eta_\ell=(\eta_{\ell,1},\ldots,\eta_{\ell,m})\in \Z^m.
%\]
%
Let us now fix a rectifiable 0-current $\mathcal{B}$ on $\R^d$ with coefficients in $\Z^m$, which is the boundary of a 1-dimensional integral $\Z^m$-current. In particular $\mathcal{B}$ is represented by the discrete $\R^n$-valued measure 
\begin{equation}\label{bordo}
\mathcal{B}=\sum_{\ell=1}^M \eta_\ell\delta_{p_\ell},
\end{equation}
where $p_\ell$ are points in $\R^d$, $\eta_\ell=(\eta_{\ell,1},\ldots,\eta_{\ell,m})\in \Z^m$ and $\sum_{\ell=1}^M \eta_\ell=(0,\dots,0)\in\Z^m$. 

If we read the problem as an optimization problem for the transportation of different goods among factories, the interpretation of ${\mathcal{B}}$ as a given datum should be the following. At each of the $M$ points $p_\ell$ a certain amount of some of the $m$ materials is produced or requested. A negative sign in the coefficient $\eta_{\ell,i}$ represents the fact that an amount $|\eta_{\ell,i}|$ of the material indexed by $i$ is produced by the factory located at the point $p_\ell$, while a positive sign represents the fact that the corresponding amount is requested by that factory.\\

%\begin{remark}
%In the definition of boundary we are allowing the situation for which a source of some material is also the sink for another material.
%\end{remark}

We are now able to state the \emph{multi-material transport problem}. Let $\cost$ satisfy properties {\rm(i), (ii), (iii)} of Definition \ref{def:multi-material_cost}.
\vspace{0.4cm}
\begin{itemize}
\item[MMTP:] Among all 1-dimensional integral $\Z^m$-currents $T=\llbracket\Sigma,\tau,\theta\rrbracket$ in $\R^d$ such that $\partial T=\mathcal{B}$, find one which minimizes the energy $\en(T)$. 
\end{itemize}

\vspace{0.4cm}
\begin{theorem}[Existence of solutions]\label{t:exist}
The MMTP admits a solution.
\end{theorem}
\begin{proof}
{The proof goes through the Direct Method of the Calculus of Variations}. The lower semicontinuity of the functional $\en$ is stated in \cite[\S 6]{White1} (see also \cite{Co_DeRo_Mar_Stu}). In \cite{Mar_Mas_Stu_Ti}, we prove the lower semicontinuity in a more general framework in order to obtain the existence of solutions to a ``continuous'' version of the MMTP. The only issue is that a minimizing sequence $\{T^l\}_{l\in\N}$ for the MMTP does not necessarily have equi-bounded masses, hence it is not possible to apply Theorem \ref{thm:cptness} directly to obtain a minimizer. Nevertheless one can obtain a uniform bound on the masses ``removing the cycles'' from the $T^l$'s. Namely, writing each $T^l=\llbracket\Sigma^l,\tau^l,\theta^l \rrbracket$ according to Theorem \ref{thm:stru} as 
$$T^l=\sum_{k=1}^{M(l)} \tilde {(T^l)}^k+\sum_{h=1}^\infty \mathring{(T^l)}^h,$$
and denoting $\tilde{T^l}=\llbracket\tilde\Sigma^l,\tau^l,\tilde\theta^l \rrbracket:=\sum_{k=1}^{M(l)} \tilde {(T^l)}^k$,
we get by \eqref{boundmultipl2} that for every $l\in\N$ it holds $\tilde \theta^l\preceq\tilde \theta^l$ $\Haus^1$-a.e. on $\Sigma^l$, hence, by the monotonicity of $\cost$, we deduce that $\en(\tilde{T^l})\leq \en(T^l)$. Observing that $\mathcal{B}=\partial T^l=\partial \tilde{T^l}$, we have that $\{\tilde{T^l}\}_{l\in\N}$ is also a minimizing sequence. Moreover, {\eqref{boundmultipl} implies that $\|\tilde\theta^l\|$ is bounded by a uniform constant $M$, $\Haus^1$ a.e. on $\tilde\Sigma^l$, and by properties (i) and (ii) of Definition \ref{def:multi-material_cost} the cost of each non-zero element of $\Z^m$ is bounded from below by a value $P$. Hence the ratio $\sfrac{\|\tilde\theta^l\|}{\cost(\tilde\theta^l)}$ is bounded by $C:=\frac{M}{P}$ $\Haus^1$ a.e. on $\tilde\Sigma^l$. Integrating on $\tilde\Sigma^l$, we deduce that $\Mass(\tilde{T^l})\leq C \en(\tilde{T^l})$. } This allows to apply Theorem \ref{thm:cptness} and to find a subsequential limit, which is a solution to the MMTP.
\end{proof}

We remark here that, under the additional assumption ${\rm (iii')}$ on the cost functional, the existence is also a trivial consequence Theorem \ref{main} below.\\

The main result of the paper is the fact that, with the additional assumption ${\rm (iii')}$ on the cost functional, the MMTP is equivalent to the superposition of a certain number of mass minimization problems among 1-dimensional integral currents, with coefficients in a group (which is larger than $\Z^m$). Introducing such problems requires some additional notation.

Let $\mathcal{B}$ be as in \eqref{bordo}. For $i=1,\dots, m$, let 
\begin{equation}\label{def_enns}
N_i:=\frac{1}{2}\sum_{\ell=1}^M|\eta_{\ell,i}|\quad \mbox{and let \quad $N:=\sum_{i=1}^m N_i$}.
\end{equation}

Note that, since $|\eta_{\ell,i}| \in\mathbb{N}$, $\forall \ell,i$, then $N_i$ are natural numbers, for every $i$, {and so is $N$}. Since $i$ represents an index for the $m$ different types of materials, then $N_i$ should be thought as the total amount of the production of the material corresponding to the index $i$ among all factories. Similarly $N$ represents the total amount of production of the union of all materials. We associate to $\mathcal{B}$ a rectifiable $0$-current, with coefficients in $\Z^N$ with the following procedure. Heuristically, for every $i=1,\dots, m$ we will give different labels to each of the $N_i$ copies produced of the $i$-th material, so that in total we will have $N$ different labels. Note that this is in a certain sense an ``unnatural'' operation, since in the original problem different copies of the same material are indistinguishable. Let us explain firstly how we assign the labels. We begin by splitting the set $\{1,\dots,N\}$ into an ordered sequence made by the $m$ groups $\{1,\dots,N_1\}$, $\{N_1+1,\dots,N_1+N_2\},\dots,\{N-N_m+1,\dots,N\}$. We will use the $i$-th group ($i=1,\dots,m$) as the set of labels for the $N_i$ copies produced of the $i$-th material. Hence to every index $j\in\{1,\dots,N\}$ we associate the corresponding $i(j)$ which is describing to which of the $m$ materials the label $j$ corresponds.
%
%Firstly we consider two vectors $P:=(P_1,\dots,P_N)$ and $D:=(D_1,\dots,D_N)$ in $\R^{Nd}$, defined via the procedure explained below.
Formally, for every $j\in\{1,\dots,N\}$, we denote $i(j)$ the first index $i$ such that $N_1+\dots+N_i\geq j$. 
Now we want to identify one of the points $\{p_\ell\}_{\ell=1}^M$ in which we will think that the copy of the $i(j)$-th material, labeled with $j$, is produced. Therefore we let $\bar j:=\sum_{k=0}^{i(j)-1}N_k$ (observe that $j-\bar j$ describes the position of the index $j$ in the $i(j)$-th group defined above) and moreover we let $\ell^-(j)$ be the first index $\ell$ such that 
$$\sum_{\eta_{\ell,i(j)}<0} |\eta_{\ell,i(j)}|\geq j-\bar j.$$

Similarly, let $\ell^+(j)$ be the first index $\ell$ such that 
$$\sum_{\eta_{\ell,i(j)}>0} \eta_{\ell,i(j)}\geq j-\bar j.$$

Finally we define 
$$\mbox{$P(j):=p_{\ell^-(j)}$ \quad and \quad $D(j):=p_{\ell^+(j)}$.}$$ 
Since it is too restrictive to assume that, for every $j=1,\dots, N$, the $i(j)$-th material produced in $P(j)$ will be sent to the point $D(j)$, we need to allow some ``reshuffling''. To this aim, we let $\sigma=(\sigma_1,\dots,\sigma_m)\in\mathcal S_{N_1}\times\dots\times\mathcal S_{N_m}$, where $\mathcal S_{q}$ is the group of permutations on $q$ elements. With a small abuse of notation, we write $\sigma(j)$ for the number $\bar j+\sigma_{i(j)}(j-\bar j)$. Note that each $\sigma_i$ ($i=1,\dots,m$) is thought as a permutation acting on the $i$-th group defined above.

Lastly we define our rectifiable 0-current, with coefficients in $\Z^N$ as
\begin{equation}\label{bordo_nuovo}
{\mathcal{B}}_\sigma:=-\sum_{j=1}^N e_j\delta_{P_j}+\sum_{j=1}^N e_j\delta_{D_{\sigma(j)}}.
\end{equation}

Observe that every fixed permutation prescribes in which of the $M$ points the copy of each labeled material will be moved. Since in our transportation problem it is not natural to prescribe such assignments, we will let the permutation vary.

Now we can state our alternative formulation of the multi-material transport problem, which is simply a \emph{mass-minimization problem}:
\begin{itemize}
\item[MMP:] Let $\|\cdot\|$ be a norm on $\R^N$. Among all $\sigma\in\mathcal S_{N_1}\times\dots\times\mathcal S_{N_m}$ and among all 1-dimensional integral $\Z^N$-currents $\bar T=\llbracket\bar\Sigma,\bar\tau,\bar\theta\rrbracket$ in $\R^d$ such that $\partial\bar T=\mathcal{B}_\sigma$ (defined in \eqref{bordo_nuovo}), find one which minimizes the mass $\Mass(\bar T)$, where the mass is computed with respect to the norm $\|\cdot\|$. 
\end{itemize}

The main result of the paper is the following
\begin{theorem}[Equivalence between MMTP and {MMP}]\label{main}
Let $\mathcal{B}$ be as in \eqref{bordo} and $N$ as in \eqref{def_enns}. Then, for every $\cost$ as in Definition \ref{def:multi-material_cost}, satisfying {\rm (i),(ii),${\rm (iii')}$}, there exists a norm $\|\cdot\|$ on $\R^N$ such that the problems MMTP and MMP are equivalent. Namely the minima are the same and moreover there is a canonical way to construct a solution of the MMTP from a solution of the MMP and vice versa.
\end{theorem}

\begin{remark}[Irrigation-type problems]
{\rm A corollary of the proof of Theorem \ref{main} is the following. If there exists one index $\bar\ell\in\{1,\dots,M\}$ such that $|\eta_{\bar\ell,i}|=\sum_{\ell\neq\bar\ell} |\eta_{\ell,i}|$, for every $i=1,\dots,m$, then in the MMP it is not necessary to minimize among the permutations $\sigma$ (i.e. the minimum is the same for every permutation). In the ``single-material'' case, the assumption corresponds to the case called ``irrigation problem'', where the initial (or the target) measure is a Dirac delta.}
\end{remark}

\begin{remark}[MMP as a Lagrangian formulation of the MMTP]
{\rm Recalling the interpretation of the coordinates of $\R^N$ as labels for different copies of the $m$ materials, we can view the MMP as a version of the MMTP where one can trace the trajectory of every particle of each type of material. The equivalence between Eulerian formulations (describing the flow of particles at every point) and Lagrangian formulations (describing the particles' trajectories) of branched transportation problem is an interesting problem in general (see e.g. \cite{Bran_WirthUP}) which is based on a profound result of Smirnov on the structure of classical normal 1-currents (see \cite{Smir}). For our discrete problem, instead, the equivalence is a simple consequence of Theorem \ref{thm:stru}. For the sake of brevity, we will not pursue this in the present paper.}
\end{remark}

%admits a solution, we could use the direct methods of Calculus of Variations. Nevertheless, in the following section we will suggest another way of solving this problem, as stated in Theorem \ref{Emmt_pb}. First, we find conditions on $\mathcal{C}$ such that we can associate to $\mathcal{C}$ a norm on $\mathbb{R}^N$, denoted in this section as $\|\cdot\|_\mathcal{C}$. Then, the boundary $B$ is translated into a boundary $B'$ with coefficients in $\mathbb{R}^N$. Finally, we prove that the minimization problem is equivalent to a mass-minimization problem for currents with coefficients in $\mathbb{R}^N$, namely the following:
%\begin{itemize}
%\item[(P2)]\label{mmt_pb}Among all integral currents $S=\llbracket\Sigma,\tau,\theta\rrbracket$ in $\R^d$ with coefficients in $\mathbb{R}^N$ such that $\partial S=B'$, find one which minimizes the mass functional.
%\end{itemize}
%\begin{comment}
%Statement of the problem and {\color{blue} NO} lower semicontinuity of the energy {\color{blue} because we deduce the existence of the minimizer through the equivalence with the mass minimization problem (the mass is l.s.c. and these integral currents can be treated componentwise)}. And many many examples (lions and giraffes, beers and pizzas and ice cream...)
%\end{comment}
\section{Equivalence between MMTP and MMP}\label{sec:3}%ricordarsi di menzionare la convessificazione e dedurre l'esistenza

The aim of this section is to establish the equivalence between the MMTP and the MMP of Section \ref{sec:2}. This follows from Theorem \ref{thm:stru}, once we find a norm $\|\cdot\|$ on $\R^N$ which is monotone and satisfies, {for every $\sigma \in \mathcal{S}_{N_1}\times\dots\times\mathcal S_{N_m}$,}
\begin{equation}\label{eqn_main}
\cost(\theta_1,\dots,\theta_m)=\Big\|\sign{(\theta_1)}\sum_{j=1}^{|\theta_1|}e_{{\sigma(j)}}+\sign{(\theta_2)}\sum_{j=N_1+1}^{N_1+|\theta_2|}e_{{\sigma(j)}}+\dots+\sign{(\theta_m)}\sum_{j=N-N_{m}+1}^{N-N_{m}+|\theta_m|}e_{{\sigma(j)}}\Big\|,
\end{equation}
where $N$ and $N_i$ $(i=1,\dots, m)$ are defined in \eqref{def_enns}. The existence of such norm would imply that the \emph{cost} of the transportation of a vector of materials $(\theta_1,\dots,\theta_n)$ along a stretch of the network corresponds to the \emph{mass} of the current with coefficients in $\R^N$ that we will associate to that stretch of the network.

The existence of such norm is proved in Theorem \ref{thm:norm}. Firstly we show how to prove Theorem \ref{main} using the existence of $\|\cdot\|$.
\begin{proof}[Proof of Theorem \ref{main}] Fix $\mathcal{B}$ as in \eqref{bordo}. We divide the proof in two steps.\\

{\bf{Step 1: from MMTP to MMP.}} Let $T:=\llbracket\Sigma,\tau,\theta\rrbracket$ be a 1-dimensional integral $\Z^m$-current which is a competitor for the MMTP. The aim of this step is to construct from $T$ a competitor $\bar T$ for the MMP ``associated'' to $\mathcal{B}$, such that $\Mass(\bar T)\leq \en (T)$. 

Consider the components of $T$ {(see Figure \ref{a})} 
$$T_1:=\llbracket\Sigma,\tau,\theta_1\rrbracket,\dots,T_m:=\llbracket\Sigma,\tau,\theta_m\rrbracket.$$ 

\begin{figure}[h]
\scalebox{0.75}{
\begin{tikzpicture}
   \begin{axis}[axis x line=none,axis y line=none,xmin=-4.8,xmax=3.5,
    ymin=-4,ymax=4,]

\addplot[]coordinates{
(-1,2) (1,2) (0,3) (-1,2)
};    

\addplot[->]coordinates{
(1,2) (0,2)
}node[below]{$(1,0)$};
    
      \addplot[]coordinates{
(-3,3) (-2,1.5) (-1,0) (1,0)};

 \addplot[]coordinates{
(-3,1.5) (-2,1.5) (-1,0) (1,0)};

 \addplot[]coordinates{
(-3,-2) (-1,0) (1,0)};

 \addplot[]coordinates{
(1,0) (2,2.5)};

 \addplot[]coordinates{
(1,0) (2,0)};

 \addplot[]coordinates{
(1,0) (2,-1.5)};

\addplot[]coordinates{
(-3,3)}
node[circle,fill=black,inner sep=0pt,minimum size=3pt]{};
\addplot[]coordinates{
(-3,3)}
node[left]{$(-1,0)$};

\addplot[]coordinates{
(-3,1.5)}
node[circle,fill=black,inner sep=0pt,minimum size=3pt]{};
\addplot[]coordinates{
(-3,1.5)}
node[left]{$(-1,0)$};

\addplot[]coordinates{
(-3,-2)}
node[circle,fill=black,inner sep=0pt,minimum size=3pt]{};
\addplot[]coordinates{
(-3,-2)}
node[left]{$(0,-1)$};

\addplot[]coordinates{
(2,2.5)}
node[circle,fill=black,inner sep=0pt,minimum size=3pt]{};
\addplot[]coordinates{
(2,2.5)}
node[right]{$(0,1)$};

\addplot[]coordinates{
(2,0)}
node[circle,fill=black,inner sep=0pt,minimum size=3pt]{};
\addplot[]coordinates{
(2,0)}
node[right]{$(1,0)$};

\addplot[]coordinates{
(2,-1.5)}
node[circle,fill=black,inner sep=0pt,minimum size=3pt]{};
\addplot[]coordinates{
(2,-1.5)}
node[right]{$(1,0)$};

\addplot[->]coordinates{(-2,1.5)(-1.5,0.75)}
node[left]{$(2,0)$};

\addplot[->]coordinates{
(-1/2,0) (0,0)}
node[below]{$(2,1)$};

   \end{axis} 
\end{tikzpicture}
\begin{tikzpicture}
   \begin{axis}[axis x line=none,axis y line=none,xmin=-4,xmax=3,
    ymin=-4,ymax=4,]
    
      \addplot[]coordinates{
(-3,3) (-2,1.5) (-1,0) (1,0) (2,-1.5)};

 \addplot[dotted]coordinates{
(-3,1.5) (-2,1.5) (-1,0) (1,0)};

 \addplot[dotted]coordinates{
(-3,-2) (-1,0) (1,0)};

 \addplot[dotted]coordinates{
(1,0) (2,2.5)};

 \addplot[dotted]coordinates{
(1,0) (2,0)};

 \addplot[dotted]coordinates{
(1,0) (2,-1.5)};

\addplot[]coordinates{
(-3,3)}
node[circle,fill=black,inner sep=0pt,minimum size=3pt]{};
\addplot[]coordinates{
(-3,3)}
node[left]{$-1$};

\addplot[]coordinates{
(-3,1.5)}
node[circle,fill=black,inner sep=0pt,minimum size=3pt]{};
\addplot[]coordinates{
(-3,1.5)}
node[left]{$-1$};

\addplot[]coordinates{
(-3,-2)}
node[circle,fill=black,inner sep=0pt,minimum size=3pt]{};

\addplot[]coordinates{
(2,2.5)}
node[circle,fill=black,inner sep=0pt,minimum size=3pt]{};

\addplot[]coordinates{
(2,0)}
node[circle,fill=black,inner sep=0pt,minimum size=3pt]{};
\addplot[]coordinates{
(2,0)}
node[right]{$1$};

\addplot[]coordinates{
(2,-1.5)}
node[circle,fill=black,inner sep=0pt,minimum size=3pt]{};
\addplot[]coordinates{
(2,-1.5)}
node[right]{$1$};

\addplot[dotted,->]coordinates{
(-2,1.5) (-1.5,0.75)};

\addplot[dotted,->]coordinates{
(1,0) (1.5,0)};

\addplot[]coordinates{
(1,0) (2,0)};

\addplot[]coordinates{
(-3,1.5) (-2,1.5)};

\addplot[->]coordinates{
(-3,1.5) (-2.5,1.5)};

\addplot[->]coordinates{
(1,0) (1.5,-0.75)};

\addplot[dotted,->]coordinates{
(-1/2,0) (0,0)}
node[above]{$T_1$};

\addplot[]coordinates{
(-1,2) (1,2) (0,3) (-1,2)
};    

\addplot[->]coordinates{
(1,2) (0,2)
}node[below]{$1$};

\addplot[->]coordinates{(-2,1.5)(-1.5,0.75)}
node[left]{$2\,$};

\addplot[->]coordinates{
(-1/2,0) (0,0)}
node[below]{$2$};

   \end{axis} 
\end{tikzpicture}
\begin{tikzpicture}
   \begin{axis}[axis x line=none,axis y line=none,xmin=-4,xmax=3,
    ymin=-4,ymax=4,]

\addplot[dotted]coordinates{
(-1,2) (1,2) (0,3) (-1,2)
};

      \addplot[dotted]coordinates{
(-3,3) (-2,1.5) };

 \addplot[dotted]coordinates{
(-3,1.5) (-2,1.5) (-1,0) (1,0)};

 \addplot[]coordinates{
(-3,-2) (-1,0) (1,0) (2,2.5)};

 \addplot[dotted]coordinates{
(1,0) (2,2.5)};

 \addplot[dotted]coordinates{
(1,0) (2,0)};

 \addplot[dotted]coordinates{
(1,0) (2,-1.5)};

\addplot[]coordinates{
(-3,3)}
node[circle,fill=black,inner sep=0pt,minimum size=3pt]{};

\addplot[]coordinates{
(-3,1.5)}
node[circle,fill=black,inner sep=0pt,minimum size=3pt]{};

\addplot[]coordinates{
(-3,-2)}
node[circle,fill=black,inner sep=0pt,minimum size=3pt]{};
\addplot[]coordinates{
(-3,-2)}
node[left]{$-1$};

\addplot[]coordinates{
(2,2.5)}
node[circle,fill=black,inner sep=0pt,minimum size=3pt]{};
\addplot[]coordinates{
(2,2.5)}
node[right]{$1$};

\addplot[]coordinates{
(2,0)}
node[circle,fill=black,inner sep=0pt,minimum size=3pt]{};

\addplot[]coordinates{
(2,-1.5)}
node[circle,fill=black,inner sep=0pt,minimum size=3pt]{};

\addplot[->]coordinates{
(-3,-2) (-2,-1)};

\addplot[->]coordinates{
(1,0) (1.75,1.875)};

\addplot[dotted,->]coordinates{
(-1/2,0) (0,0)}
node[above]{$T_2$};

\addplot[->]coordinates{
(-1/2,0) (0,0)}
node[below]{$1$};

   \end{axis} 
\end{tikzpicture}}
\caption{On the left a 1-dimensional integral $\Z^2$-current $T=(T_1,T_2)$ and on the right its two components.}
\label{a}
\end{figure}
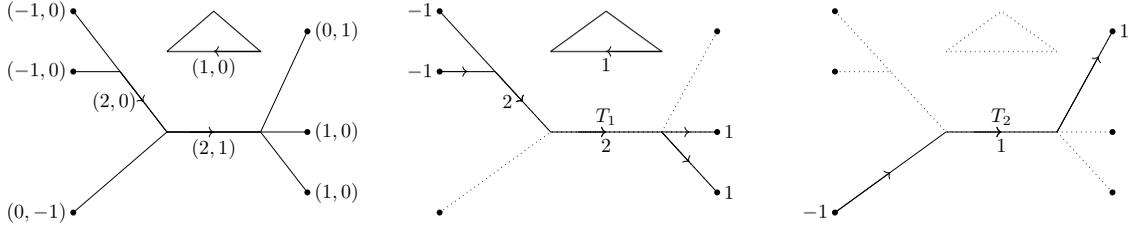

By \cite[
4.2.25]{Fe1} we can write, for $i=1,\dots,m$ 
\begin{equation}\label{eqn_struttura}
T_i=\sum_{k=1}^{N_i} \tilde T_i^k+\sum_{h=1}^\infty \mathring{T}_i^h,
\end{equation}
with\footnote{{Observe that all the currents in \eqref{eqn_spezzamassa} are ``classical'' currents, therefore the notion of mass is the standard one.}} 
\begin{equation}\label{eqn_spezzamassa}
\mathbb{M}(T_i)=\sum_{k=1}^{N_i} \mathbb{M}(\tilde T_i^k)+\sum_{h=1}^\infty \mathbb{M}(\mathring{T}_i^h)\quad \mbox{ and \quad $\mathbb{M}(\partial T_i)=\sum_{k=1}^{N_i} \mathbb{M}(\partial \tilde T_i^k)$},
\end{equation}
where:
\begin{itemize}
\item[$\circ$] $\tilde T_i^k:=\llbracket\Gamma_i^k,\tau_i^k,1\rrbracket$ are integral $1$-currents associated to simple, Lipschitz, open curves $\gamma_i^k:[0,1]\to\R^d$, where $\Gamma_i^k:=\mbox{Im}(\gamma_i^k)$ and $\tau_i^k:=\frac{(\gamma_i^k)'}{|(\gamma_i^k)'|}$;
\item[$\circ$] $\mathring{T}_i^h:=\llbracket Z_i^h,\nu_i^h,1\rrbracket$ are integral $1$-currents associated to simple, Lipschitz, closed curves (cycles) $\zeta_i^h:[0,1]\to\R^d$, where $Z_i^h:=\mbox{Im}(\zeta_i^h)$ and $\nu_i^h:=\frac{(\zeta_i^h)'}{|(\zeta_i^h)'|}$.
\end{itemize}

For every $i=1,\dots,m$, denote by $\Sigma_i:=\cup_{k=1}^{N_i}\Gamma_i^k$ and by $T'$ the $1$-current 
$$T_i':= \sum_{k=1}^{N_i} \tilde T_i^k.$$
Since each $T_i'$ is supported on $\Sigma$, we can write $T_i':=\llbracket\Sigma,\tau,\theta_i'\rrbracket$. Let $T'$ be the 1-dimensional integral $\Z^m$-current whose components are $(T'_1,\dots,T'_m)$. 

It follows from \eqref{eqn_struttura} that for every $i=1,\dots,m$ and for $\Haus^1$-a.e. $x\in\Sigma$ it holds 
\begin{equation}\label{eqn_multipl}
\theta_i(x)=\sum_{k=1}^{N_i} \chi_{\Gamma_i^k}(x)\langle\tau_i^k(x);\tau(x)\rangle+\sum_{h=1}^\infty \chi_{Z_i^h}(x)\langle\nu_i^h(x);\tau(x)\rangle,
\end{equation}
where we denoted by $\chi_E$ the characteristic function of the set $E$ taking values $0$ and $1$. 
 Combining \eqref{eqn_multipl} and \eqref{eqn_spezzamassa} we deduce that for every $i=1,\dots, m$ it holds 

$$\tau_i^k(x)=\sign(\theta_i(x))\tau(x), \quad \mbox{for $\Haus^1$-a.e. $x\in\Gamma_i^k$, for every $k$}$$
and
$$\nu_i^h(x)=\sign(\theta_i(x))\tau(x), \quad \mbox{for $\Haus^1$-a.e. $x\in Z_i^h$, for every $h$}.$$ 
Hence it holds $(\theta_1'(x),\dots,\theta_m'(x))\preceq(\theta_1(x),\dots,\theta_m(x))$, for $\Haus^1$-a.e. $x\in\Sigma$, which yields $\en(T')\leq\en(T)$, by property (ii) of Definition \ref{def:multi-material_cost}. Moreover by \eqref{eqn_struttura} it holds $\partial T'=\partial T$.

Next we associate to $T'$ a 1-dimensional integral $\Z^N$-current $\bar T$, simply defining $\bar T$ to be the current with components $(\bar T_1,\dots, \bar T_N)$, where we set, for $j=1,\dots, N$ (recalling the definition of $i(j)$ and $\bar j$ from \S\ref{sec:2}) $\bar T_j:=\tilde T_{i(j)}^k, \mbox{ for $k=j-\bar j$}.$
\begin{figure}[h]
\scalebox{0.75}{
\begin{tikzpicture}
   \begin{axis}[axis x line=none,axis y line=none,xmin=-4,xmax=3,
    ymin=-4,ymax=4,]
    
      \addplot[]coordinates{
(-3,3) (-2,1.5) (-1,0) (1,0) (2,-1.5)};

 \addplot[dotted]coordinates{
(-3,1.5) (-2,1.5) (-1,0) (1,0)};

 \addplot[dotted]coordinates{
(-3,-2) (-1,0) (1,0)};

 \addplot[dotted]coordinates{
(1,0) (2,2.5)};

 \addplot[dotted]coordinates{
(1,0) (2,0)};

 \addplot[dotted]coordinates{
(1,0) (2,-1.5)};

\addplot[]coordinates{
(-3,3)}
node[circle,fill=black,inner sep=0pt,minimum size=3pt]{};
\addplot[]coordinates{
(-3,3)}
node[left]{$-1$};

\addplot[]coordinates{
(-3,1.5)}
node[circle,fill=black,inner sep=0pt,minimum size=3pt]{};

\addplot[]coordinates{
(-3,-2)}
node[circle,fill=black,inner sep=0pt,minimum size=3pt]{};

\addplot[]coordinates{
(2,2.5)}
node[circle,fill=black,inner sep=0pt,minimum size=3pt]{};

\addplot[]coordinates{
(2,0)}
node[circle,fill=black,inner sep=0pt,minimum size=3pt]{};

\addplot[]coordinates{
(2,-1.5)}
node[circle,fill=black,inner sep=0pt,minimum size=3pt]{};
\addplot[]coordinates{
(2,-1.5)}
node[right]{$1$};

\addplot[dotted,->]coordinates{
(-2,1.5) (-1.5,0.75)};

\addplot[->]coordinates{
(1,0) (1.5,-0.75)};

\addplot[dotted,->]coordinates{
(-1/2,0) (0,0)}
node[above]{$\bar T_1$};

   \end{axis} 
\end{tikzpicture}

\begin{tikzpicture}
   \begin{axis}[axis x line=none,axis y line=none,xmin=-4,xmax=3,
    ymin=-4,ymax=4,]
    
      \addplot[dotted]coordinates{
(-3,3) (-2,1.5) };

 \addplot[]coordinates{
(-3,1.5) (-2,1.5) (-1,0) (1,0)};

 \addplot[dotted]coordinates{
(-3,-2) (-1,0)};

 \addplot[dotted]coordinates{
(-3,-2) (-1,0) (1,0) (2,2.5)};

 \addplot[dotted]coordinates{
(1,0) (2,2.5)};

 \addplot[dotted]coordinates{
(1,0) (2,0)};

 \addplot[dotted]coordinates{
(1,0) (2,-1.5)};

\addplot[]coordinates{
(-3,3)}
node[circle,fill=black,inner sep=0pt,minimum size=3pt]{};

\addplot[]coordinates{
(-3,1.5)}
node[circle,fill=black,inner sep=0pt,minimum size=3pt]{};
\addplot[]coordinates{
(-3,1.5)}
node[left]{$-1$};

\addplot[]coordinates{
(-3,-2)}
node[circle,fill=black,inner sep=0pt,minimum size=3pt]{};

\addplot[]coordinates{
(2,2.5)}
node[circle,fill=black,inner sep=0pt,minimum size=3pt]{};

\addplot[]coordinates{
(2,0)}
node[circle,fill=black,inner sep=0pt,minimum size=3pt]{};
\addplot[]coordinates{
(2,0)}
node[right]{$1$};

\addplot[]coordinates{
(2,-1.5)}
node[circle,fill=black,inner sep=0pt,minimum size=3pt]{};

\addplot[->]coordinates{
(-3,1.5) (-2.5,1.5)};

\addplot[->]coordinates{
(-2,1.5) (-1.5,0.75)};

\addplot[]coordinates{
(1,0) (2,0)};

\addplot[dotted,->]coordinates{
(-1/2,0) (0,0)}
node[above]{$\bar T_2$};

   \end{axis} 
\end{tikzpicture}
\begin{tikzpicture}
   \begin{axis}[axis x line=none,axis y line=none,xmin=-4,xmax=3,
    ymin=-4,ymax=4,]
    
      \addplot[dotted]coordinates{
(-3,3) (-2,1.5) };

 \addplot[dotted]coordinates{
(-3,1.5) (-2,1.5) (-1,0) (1,0)};

 \addplot[]coordinates{
(-3,-2) (-1,0) (1,0) (2,2.5)};

 \addplot[dotted]coordinates{
(1,0) (2,2.5)};

 \addplot[dotted]coordinates{
(1,0) (2,0)};

 \addplot[dotted]coordinates{
(1,0) (2,-1.5)};

\addplot[]coordinates{
(-3,3)}
node[circle,fill=black,inner sep=0pt,minimum size=3pt]{};

\addplot[]coordinates{
(-3,1.5)}
node[circle,fill=black,inner sep=0pt,minimum size=3pt]{};

\addplot[]coordinates{
(-3,-2)}
node[circle,fill=black,inner sep=0pt,minimum size=3pt]{};
\addplot[]coordinates{
(-3,-2)}
node[left]{$-1$};

\addplot[]coordinates{
(2,2.5)}
node[circle,fill=black,inner sep=0pt,minimum size=3pt]{};
\addplot[]coordinates{
(2,2.5)}
node[right]{$1$};

\addplot[]coordinates{
(2,0)}
node[circle,fill=black,inner sep=0pt,minimum size=3pt]{};

\addplot[]coordinates{
(2,-1.5)}
node[circle,fill=black,inner sep=0pt,minimum size=3pt]{};

\addplot[->]coordinates{
(-3,-2) (-2,-1)};

\addplot[->]coordinates{
(1,0) (1.75,1.875)};

\addplot[dotted,->]coordinates{
(-1/2,0) (0,0)}
node[above]{$\bar T_3$};

   \end{axis} 
\end{tikzpicture}}
\caption{The components $\bar T_1, \bar T_2,$ and $\bar T_3$ of the integral $\Z^3$-current $\bar T$.}
\label{b}
\end{figure}
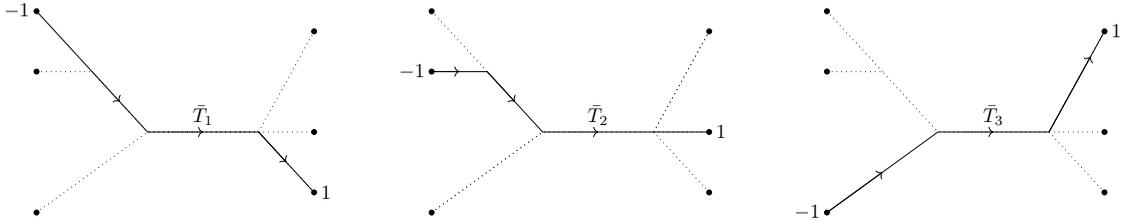

Applying the boundary operator to \eqref{eqn_struttura}, it follows that $\bar T$ is a competitor for the MMP. Moreover by \eqref{eqn_multipl} and \eqref{eqn_main} it follows that $\Mass(\bar T)=\en(T')\leq\en(T)$.\\

{\bf{Step 2: from MMP to MMTP.}} Let $\sigma\in\mathcal S_{N_1}\times\dots\times\mathcal S_{N_m}$. Let $\bar T:=\llbracket\bar\Sigma,\bar\tau,\bar\theta\rrbracket$ be a 1-dimensional integral $\Z^N$-current which is a competitor for the MMP and in particular $\partial \bar T=\mathcal{B}_\sigma$ (defined in \eqref{bordo_nuovo}). The aim of this step is to construct from $\bar T$ a competitor $T$ for the MMTP associated to $\mathcal{B}$, such that $\en(T)\leq \Mass (\bar T)$.

Let 
$$\bar T_1:=\llbracket\bar\Sigma,\bar\tau,\bar\theta_1\rrbracket,\dots, \bar T_N:=\llbracket\bar\Sigma,\bar\tau,\bar\theta_N\rrbracket$$ 
be the components of $\bar T$ {(see Figure \ref{b})}.  

As in the previous step, by \cite[4.2.25]{Fe1} and using the fact that $\Mass(\partial \bar T_j)=2$, we can write for $j=1,\dots,N$ 
\begin{equation}\label{eqn_struttura2}
\bar T_j= \tilde T_j+\sum_{h=1}^\infty \mathring{T}_j^h,
\end{equation}
with 
\begin{equation}\label{eqn_spezzamassa2}
\mathbb{M}(\bar T_j)=\mathbb{M}(\tilde T_j)+\sum_{h=1}^\infty \mathbb{M}(\mathring{T}_j^h),
\end{equation}
where:
\begin{itemize}
\item[$\circ$] $\tilde T_j:=\llbracket\Gamma_j,\tau_j,1\rrbracket$ are integral $1$-currents associated to simple, Lipschitz, open curves $\gamma_j:[0,1]\to\R^d$, where $\Gamma_j:=\mbox{Im}(\gamma_j)$ and $\tau_j:=\frac{(\gamma_j)'}{|(\gamma_j)'|}$;
\item[$\circ$] $\mathring{T}_j^h:=\llbracket Z_j^h,\nu_j^h,1\rrbracket$ are integral $1$-currents associated to simple Lipschitz closed curves $\zeta_j^h:[0,1]\to\R^d$, with $\zeta(0)=\zeta(1)$, where $Z_j^h:=\mbox{Im}(\zeta_j^h)$ and $\nu_j^h:=\frac{(\zeta_j^h)'}{|(\zeta_j^h)'|}$.
\end{itemize}
  Let $T'$ be the 1-dimensional integral $\Z^N$-current whose components are $(\tilde T_1,\dots,\tilde T_N)$. By \eqref{eqn_struttura2} and \eqref{eqn_spezzamassa2}, for $j=1,\dots, N$ it holds $\langle\bar\tau;\tilde\tau_j\rangle=\sign(\bar\theta_j)$ $\Haus^{1}$-a.e. in $\tilde\Gamma_j$ and hence, since $\|\cdot\|$ is a monotone norm, we have $\Mass(T')\leq \Mass(\bar T)$. Moreover, by \eqref{eqn_struttura2} it holds $\partial T'=\partial \bar T$.

Let $T$ be the 1-dimensional integral $\Z^m$-current with components (recalling the definition of $i(j)$ from \S\ref{sec:2})
$$T_i:=\sum_{j:i(j)=i}\tilde T_j.$$

Since $\partial T'= \mathcal{B}_\sigma$, then $\partial T=\mathcal{B}$. Moreover by \eqref{eqn_main} it holds $\en(T)=\Mass(T')\leq \Mass(\bar T)$.
\end{proof}

We conclude this section by proving the existence of a monotone norm $\|\cdot\|$ satisfying \eqref{eqn_main}.

In the proof of the next theorem, we will use the following fact, which can be found in \cite{Bauer_Stoer_Witz}. Recall the notions of monotone and absolute norm given at the beginning of Section \ref{sec:2}, as well as the partial order introduced there. 

\begin{lem}\label{ABS}
An absolute norm on $\mathbb{R}^n$ is monotone.
\end{lem}

We will use the term \emph{orthant} in $\mathbb{R}^n$ for the following subset of $\mathbb{R}^n$. Consider a vector $\xi \in \mathbb{R}^d$ whose coordinates are only $\pm 1$. The $\xi$-orthant is:
\[
\{x \in \mathbb{R}^d: \xi_\ell x_\ell \geq 0, \forall \ell=1,\dots,n\}.
\]
Note that an orthant is always closed.

\begin{thm}[Existence of a norm satisfying \eqref{eqn_main}]\label{thm:norm}
Let $\cost:\Z^m\to[0,+\infty)$ be a function satisfying ${\rm (i)}$, ${\rm (ii)}$, ${\rm (iii')}$ of Definition \ref{def:multi-material_cost}. Let $\mathcal{B}$ be as in \eqref{bordo} and let $N$ and $N_i$ $(i=1,\dots, m)$ be the natural numbers defined in \eqref{def_enns}. Then there exists a monotone norm $\|\cdot\|$ on $\R^N$ satisfying \eqref{eqn_main}.
\end{thm}

\begin{proof}
\noindent{\bf Step 1}: First of all, let us suppose that $\cost$ has the additional property that
\begin{equation}\label{supersymm}
\cost(x) = \cost(\bar x)
\end{equation}
for every $x\in \Z^m$, where we used the notation introduced at the beginning of Section \ref{nota}. 

\vspace{0.2cm}
\fbox{\it General strategy}
\vspace{0.2cm}

Let us denote by $\mathcal{A}$ the set of elements of $\Z^N$ whose coordinates are only $0$'s and $1$'s, and denote $\bar{\mathcal{A}} := \{x \in \mathbb{Z}^N: \bar x\in \mathcal{A}\}$.

Let $A = (a_1,\dots,a_N)$, $B = (b_1,\dots,b_N) \in\mathcal{A}$. We say that the pair $(A,B)\in{\mathcal{A}}\times{\mathcal{A}}$ is \emph{good} if $A-B\neq 0$ and the following implications hold, for every $i=1,\dots,m$ (again, we set $N_0:=0$):
\begin{itemize}
\item[$\circ$] if $a_j=1$ for some $j$ between $N_1+\dots+N_{i-1}+1$ and $N_1+\dots+N_{i}$ then $b_h=0$ for all indices $h$ in the same range;
\item[$\circ$] if $b_j=1$ for some $j$ between $N_1+\dots+N_{i-1}+1$ and $N_1+\dots+N_{i}$ then $a_h=0$ for all indices $h$ in the same range.
\end{itemize}
Recalling the heuristic interpretation described in Section \ref{sec:2}, any vector of $\R^N$ whose coordinates are only $-1,1$, or $0$, represents a collection of labelled materials, possibly of different types. If $(A,B)$ is a good pair, {then on each of the $m$ groups into which we split the set of labels $\{1,\dots,N\}$ at most one among $A$ and $B$ can have some non-zero coordinates}. Therefore in the corresponding vector $A-B$, all the coordinates in each group { belong either to $\{0,1\}$ or to $\{0,-1\}$}. This represents the fact that all the materials of the same type are assumed to travel with the same orientation.

If $(A,B)$ is a good pair we define
\begin{equation}\label{eqn_defcab}
c_{A,B}:=\cost\left(\sum_{j=1}^{N_1}(a_j-b_j),\dots,\sum_{j=N-N_{m-1}+1}^N(a_j-b_j)\right).
\end{equation} 
{This represents the cost of transporting the collection of goods labeled by $A-B$ along a stretch of unit length.}
Observing that, if $A - B\neq 0$ we have $c_{A,B}\neq 0$, we can define
\[
 q_{A,B}:=\frac{A-B}{c_{A,B}},
\]
for any good pair $(A,B)$. Observe that if $(A,B)$ and $(A',B')$ are good pairs with $\overline{A-B}=\overline{A'-B'}$, then by \eqref{supersymm} it holds $c_{A,B}=c_{A',B'}$. Hence given $D\neq 0$ any vector in $\bar{\mathcal{A}}$, it is convenient to define $c_D:=c_{\bar D,0}$, which is well defined since $(\bar D,0)$ is a good pair. As above, we define $q_D:=\frac{D}{c_D}$. 
%Observe that by \eqref{supersymm} we get also that $c_{A.B} = c_{|\alpha|}$, being $\alpha = |A| + |B|$.
Consider the convex hull 
$$C:={\co}(\{q_D: D\in \bar{\mathcal{A}}\setminus\{0\}\})\subset \R^N.$$ 
The theorem is proven if we show three properties of $C$:
\begin{itemize}
\item[(1)] $C$ is a convex body (i.e. the closure of its non-empty interior) which is bounded and symmetric with respect to the origin;
\item[(2)] $C$ is a monotone set, i.e. for every $x,y\in\R^N$, with $y\preceq x$, if $x\in C$, then also $y\in C$.
\item[(3)] it holds
\[
q_{A,B}\in\partial C,\quad\forall A,B\in\mathcal{A},\quad \mbox{such that $(A,B)$ is a good pair}.
\]
\end{itemize} 
Indeed, if (1) holds, there exists a norm $\|\cdot\|$ on $\R^N$ whose unit ball is the set $C$. Then, (2) implies that $\|\cdot\|$ is monotone. Moreover (3) implies that $\|\cdot\|$ satisfies \eqref{eqn_main}. Indeed, take an $m$-tuple $(\theta_1,\dots,\theta_m)$ and denote, for every $i=1,\dots, m$, $\theta_i^+:=\max\{\sign{\theta_i,0}\}$, $\theta_i^-:=\max\{-\sign{\theta_i,0}\}$. Now define for every $\sigma \in \mathcal{S}_{N_1}\times\dots\times\mathcal{S}_{N_m}$,
$$A_\sigma:=\theta_1^+\sum_{j=1}^{|\theta_1|}e_{\sigma(j)}+\theta_2^+\sum_{j=N_1+1}^{N_1+|\theta_2|}e_{\sigma(j)}+\dots+\theta_m^+\sum_{j=N-N_{m}+1}^{N-N_{m}+|\theta_m|}e_{\sigma(j)}$$ 
and
$$B_\sigma:=\theta_1^-\sum_{j=1}^{|\theta_1|}e_{\sigma(j)}+\theta_2^-\sum_{j=N_1+1}^{N_1+|\theta_2|}e_{\sigma(j)}+\dots+\theta_m^-\sum_{j=N-N_{m}+1}^{N-N_{m}+|\theta_m|}e_{\sigma(j)}.$$ 
One can verify that $A_\sigma,B_\sigma\in\mathcal{A}$ and $(A_\sigma,B_\sigma)$ is a good pair. Hence we have 
$$1\stackrel{(3)}{=}\|q_{A_\sigma,B_\sigma}\|=\frac{\|A_\sigma-B_\sigma\|}{c_{A_\sigma,B_\sigma}}\stackrel{\eqref{eqn_defcab}}{=}\frac{\|A_\sigma-B_\sigma\|}{\cost(\theta_1,\dots,\theta_m)}.$$
We conclude noting that $A_\sigma-B_\sigma$ coincides with the RHS of \eqref{eqn_main}.

\vspace{0.2cm}
\fbox{\it Proof of (1) and (2)}
\vspace{0.2cm}

To prove (1), notice that for every $j=1,\dots,N$, $q_{\pm e_j}$ are contained in $C$, hence $0\in {\rm int}(C)$. The fact that $C$ is symmetric with respect to the origin follows from the fact that the multi-material cost $\cost$ is even. Finally, the boundedness is trivial, since $C$ the convex hull of a finite set. 

We will now prove (2), i.e. that $C$ is a monotone set. To prove it, we show that the norm with unit ball $C$ is absolute. This implies the monotonicity by Lemma \ref{ABS}. Let $x \in \mathbb{R}^N$, with $\|x\|=1$. The fact that $\|x\| = 1$ implies that $x \in \partial C\subset C$, therefore we can write
\[
x = \sum_{k=1}^K t_kq_{D^k},
\]
where $D^k\in\bar{\mathcal{A}}\setminus \{0\}$, $\sum_{k=1}^{K} t_k = 1$, with $t_k$ positive.
There exists a diagonal matrix $M\in{\rm Mat}^{N\times N}$ with entries $1,-1,0$ such that $Mx = \bar x$. Therefore:
\[
\|\bar x\|=\|Mx\| = \left\|\sum_{k=1}^K t_kMq_{D^k}\right\| =\left\|\sum_{k=1}^K t_k q_{MD^k}\right\|\leq \sum_{k=1}^K t_k\|q_{MD^k}\| \leq  \sum_{k=1}^K t_k\leq 1,
\]
where the third equality follows from the fact that $c_{D^k}=c_{MD^k}$ (the latter being a consequence of the fact that $\overline{MD}=\overline{D}$ for every $D\in\bar{\mathcal{A}}\setminus \{0\}$) and the second inequality follows from the fact that $\|q_{D}\| \leq 1, \forall D\in\bar{\mathcal{A}}$, by the definition of $C$. This proves that $$\|\bar{x}\|\le\|x\|, \quad\forall x\in\R^N.$$ The proof of the reverse inequality is analogous.

\vspace{0.2cm}
\fbox{\it Proof of (3)}
\vspace{0.2cm}

The proof of (3) is more involved. We can prove equivalently that for every $A,B\in\mathcal{A}$, such that $(A,B)$ is a good pair, and for every $t>0$ the following implication holds

\begin{equation}
tq_{A,B}\in C\Longrightarrow t\le 1.
\end{equation}
Since $tq_{A,B} \in C$, we can write
\begin{equation}\label{eqn_tqAB}
tq_{A,B}=\sum_{k=1}^K\lambda_kq_{D^k},
\end{equation}
where $D^k\in\bar{\mathcal{A}}$, $\sum_{k=1}^{K}\lambda_k = 1$, with $\lambda_k$ positive. Formula \eqref{eqn_tqAB} can be rewritten componentwise, denoting $D^k = (d^k_1,\dots,d^k_N)$,
\[
t\frac{a_j - b_j}{c_{A,B}}=\sum_{k}\lambda_k\frac{d^k_j}{c_{D^k}}, \text{ for every } j=1,\dots,N.
\]
For $k=1,\dots,K$, we define vectors $F^k:=(f^k_1,\dots,f^k_N)\in \bar{\mathcal{A}}$ by
\begin{equation}\label{zero}
\begin{cases}
f^k_j := 0, \text{ if $a_j - b_j$ = 0}\\
f^k_j := d^k_j, \text{ otherwise. }
\end{cases}
\end{equation}
Note that
\begin{equation}\label{eqn_tqAB_comp}
t\frac{a_j - b_j}{c_{A,B}}=\sum_{k}\lambda_k\frac{f^k_j}{c_{D^k}}, \text{ for every } j=1,\dots,N.
\end{equation}

Indeed, the equality $$\sum_{k}\lambda_k\frac{d^k_j}{c_{D^k}} = \sum_{k}\lambda_k\frac{f^k_j}{c_{D^k}}$$ holds for those $j$ such that $\sum_{k}\lambda_k\frac{d^k_j}{c_{D^k}}\neq 0$ because in that case $f_j^k = d_j^k$ for every $k$. On the other hand, for those indices $j$ for which $\sum_{k}\lambda_k\frac{d^k_j}{c_{D^k}} = 0$ by the definition of $F^k$, also $f_j^k = 0$ for every $k$, so that
\[
\sum_{k}\lambda_k\frac{f^k_j}{c_{D^k}}=0 = \sum_{k}\lambda_k\frac{d^k_j}{c_{D^k}}.
\]

\begin{comment} $f_j^k$ might differ from $d_j^k$ (and it is equal to zero for every $k$) only in correspondence of those indices $j$ where $\sum_{k}\lambda_k\frac{d^k_j}{c_{D^k}}=0$.
\end{comment}

Moreover
\begin{equation}\label{monoto}
c_{F^k} \leq c_{D^k}
\end{equation} 
by property (ii) in Definition \ref{def:multi-material_cost} (because $\bar F^k\preceq \bar D^k$ by definition of $F^k$).
%by the monotonicity of the cost. Passing to the absolute value in \eqref{eqn_tqAB_comp}, we have:
%\[
%t\frac{|a_j - b_j|}{c_{A,B}}\leq \sum_{k}\lambda_k\frac{|e^k_j|}{c_{D^k}}, \forall j.
%\]
Denote, for $i=1,\dots,m$,
\[
x_i := \sum_{j = N_1 + \dots + N_{i - 1} + 1}^{N_1 + \dots + N_{i}}a_j - b_j,
\]
and for $k=1,\dots, K$,
\[
x^k_i := \sum_{j = N_1 + \dots + N_{i - 1} + 1}^{N_1 + \dots + N_{i}}e^k_j.
\]
Define, for every $k = 1,\dots,K$ and for every $i=1,\dots,m$,
\[
y^k_i:=
\begin{cases}
x^k_i, \text{ if } x^k_ix_i \geq 0\\
-x^k_i, \text{ if } x^k_ix_i \leq 0.
\end{cases}
\]

Finally, denote $x:= (x_1,\dots,x_m)$, and $y^k:= (y_1^k,\dots,y_m^k)$, for $k=1,\dots,K$. By \eqref{zero}, we have that $y^k \preceq x$, for every $k$. To see this, note that, by the definition of $y^k$, it immediately follows that $y_i^kx_i \ge 0$. To prove that $|y_i^k|\le|x_i|$ for every $i$ and $k$, we recall the fact that $(A,B)$ is a good pair, so that, in particular, we have the following property:
\begin{equation}\label{signs}
\sum_{j = N_1 + \dots + N_{i - 1} + 1}^{N_1 + \dots + N_{i}}|a_j - b_j| = \left|\sum_{j = N_1 + \dots + N_{i - 1} + 1}^{N_1 + \dots + N_{i}}a_j - b_j\right|, \forall i.
\end{equation}
Also, since $|a_j - b_j| \in \{0,1\} $ for every $j$, by the definition of $f_j^k$, it readily follows that
\begin{equation}\label{pointwise}
|f_j^k| \le |a_j - b_j|, \forall j,k.
\end{equation}
We also have
\[
|y_i^k| = |x_i^k| \le \sum_{j = N_1 + \dots + N_{i - 1} + 1}^{N_1 + \dots + N_{i}}|f^k_j|\stackrel{\eqref{pointwise}}{\le} \sum_{j = N_1 + \dots + N_{i - 1} + 1}^{N_1 + \dots + N_{i}}|a_j - b_j| \stackrel{\eqref{signs}}{=} \left|\sum_{j = N_1 + \dots + N_{i - 1} + 1}^{N_1 + \dots + N_{i}}a_j - b_j\right| = |x_i|.
\]

Moreover, by \eqref{eqn_tqAB_comp}, for every $i=1,\dots,m$ it holds
\[
\frac{tx_i}{c_{A,B}} = \sum_k\lambda_k\frac{x_i^k}{c_{D^k}}, 
\]
hence the fact that $\sign(y_i^k) = \sign(x_i)$ implies
\begin{align*}
\frac{t|x_i|}{c_{A,B}}&=\frac{t\sign(x_i)x_i}{c_{A,B}} = \sum_k\lambda_k\frac{\sign(x_i)x_i^k}{c_{D^k}} = \sum_k\lambda_k\frac{\sign(x_i)\sign(x_i^k)|x_i^k|}{c_{D^k}} =\\
&=\sum_k\lambda_k\frac{\sign(x_ix_i^k)|y_i^k|}{c_{D^k}} = \sum_k\lambda_k\frac{\sign(x_i^k)y_i^k}{c_{D^k}}.
\end{align*}
From the equality $\sign(y_i^k) = \sign(x_i)$, holding for every $k,i$, we also deduce that, if we fix $i$, the quantity $\sign(y_i^k)$ remains constant when varying $k$.
Now, if $y_i^k$ is positive, for every $k$, since $\sign(x_i^k) \leq 1$, we get
\[
\frac{t|x_i|}{c_{A,B}} \le \sum_k\lambda_k\frac{y_i^k}{c_{D^k}} = \left|\sum_k\lambda_k\frac{y_i^k}{c_{D^k}}\right|.
\]
Otherwise, using the fact that $\sign(x_i^k) \geq -1$,
\[
\frac{t|x_i|}{c_{A,B}} \le \sum_k\lambda_k\frac{-y_i^k}{c_{D^k}} = \left|\sum_k\lambda_k\frac{y_i^k}{c_{D^k}}\right|.
\]
We have just proved that
\begin{equation}\label{stimazza}
t\frac{x}{c_{A,B}} \preceq \sum_{k:F^k\neq 0}\lambda_k\frac{y^k}{c_{D^k}}.
\end{equation}
Finally, note also that $\bar y^k = \bar x^k$. By \eqref{supersymm} it holds $c_{A,B} = \cost(x)$ and $c_{F^k} = c_{\bar F^k}= \cost(x^k) = \cost(y^k)$, this implies, by property ${\rm (iii')}$ of Definition \ref{def:multi-material_cost} that
\begin{equation}\label{daidaidai}
\frac{c_{A,B}\|y^k\|_\star}{c_{F^k}\|x\|_\star} = \frac{\cost(x)\|y^k\|_\star}{\cost(y^k)\|x\|_\star} \leq 1,\quad \forall k=1,\dots,K \mbox{ such that $F^k\neq 0$},
\end{equation}
where $\|\cdot\|_\star$ is the norm appearing in such definition. Using that $\|\cdot\|_\star$ is monotone, \eqref{monoto} and \eqref{stimazza}, we get
\[
t\|x\|_\star \le \left\|\sum_{k:F^k\neq 0}\lambda_k\frac{c_{A,B}y^k}{c_{D^k}} \right\|_\star \le \sum_{k:F^k\neq 0}\left\|\lambda_k\frac{c_{A,B}y^k}{c_{D^k}}\right\|_\star = \sum_{k:F^k\neq 0}\lambda_k\frac{c_{A,B}\|y^k\|_\star}{c_{D^k}} \le \sum_{k:F^k\neq 0}\lambda_k\frac{c_{A,B}\|y^k\|_\star}{c_{F^k}}.
\]
Finally, dividing by $\|x\|_\star$, \eqref{daidaidai} yields
\[
t \leq \sum_{k:F^k\neq 0}\lambda_k \leq 1.
\]

\noindent{\bf Step 2:}
Now consider a general cost $\cost$, which does not necessarily satisfy \eqref{supersymm}. We will construct a closed, convex, and symmetric set $C$, whose associated norm is monotone and satisfies \eqref{eqn_main}.

\vspace{0.2cm}
\fbox{\it General strategy}
\vspace{0.2cm}

For any orthant $\mathcal{O}\subset\R^m$, we define a cost $\cost_\mathcal{O}:\Z^m\to [0,+\infty)$, imposing the following properties:
\begin{itemize}
\item [(a)] $\cost_\mathcal{O}(x)=\cost(x)$ if $x\in \mathcal{O}$;
\item [(b)] $\cost_\mathcal{O}(x)=\cost_\mathcal{O}(\bar x)$ for every $x$.
\end{itemize}
Trivially properties (i),(ii),${\rm (iii')}$ of Definition \ref{def:multi-material_cost} are satisfied by $C_{\mathcal{O}}$. 

Let $\|\cdot\|_{\mathcal{O}}$ be the norm on $\R^N$ obtained applying Step 1 to the cost $\cost_\mathcal{O}$ and let $B_{\mathcal{O}}$ be the unit ball with respect to such norm {(see Figure \ref{c})}. Let us take any point $x\in{\rm{int}}(\mathcal {O})$ and define 
$$\sigma_\mathcal{O}:=(\sign(x_1),\dots,\sign(x_m))\in\R^m.$$
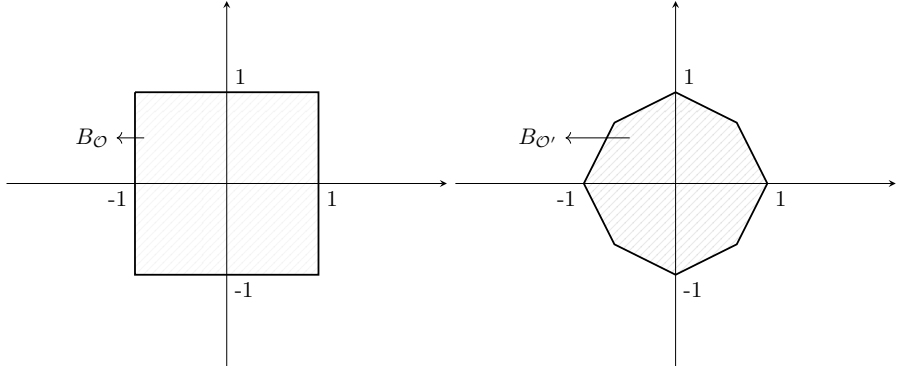
\begin{figure}[h]
\scalebox{0.85}{
\begin{tikzpicture}
   \begin{axis}[axis x line=middle,axis y line=middle,xmin=-2.4,xmax=2.4,ticks=none,
    ymin=-2,ymax=2,]
     
\addplot[thick] coordinates{
(-1,1) (1,1) (1,-1) (-1,-1) (-1,1)};
\addplot[pattern = north east lines, opacity=0.2] coordinates{
(-1,1) (1,1) (1,-1) (-1,-1) (-1,1)};   
   
   \addplot[->] coordinates{
(-0.9,0.5) (-1.2,0.5)}
node[left] {$B_{\mathcal{O}}$};

\addplot[] coordinates{
(1,0)}
node[below right] {1};
\addplot[] coordinates{
(-1,0)}
node[below left] {-1};
\addplot[] coordinates{
(0,1)}
node[above right] {1};
\addplot[] coordinates{
(0,-1)}
node[below right] {-1};
\end{axis}
   
\end{tikzpicture}

\begin{tikzpicture}
   \begin{axis}[axis x line=middle,axis y line=middle,xmin=-2.4,xmax=2.4,ticks=none,
    ymin=-2,ymax=2,]

\addplot[thick] coordinates{
(-2/3,2/3) (0,1) (2/3,2/3) (1,0) (2/3,-2/3)(0,-1)(-2/3,-2/3)(-1,0)(-2/3,2/3)};
\addplot[pattern = north east lines, opacity=0.3] coordinates{
(-2/3,2/3) (0,1) (2/3,2/3) (1,0) (2/3,-2/3)(0,-1)(-2/3,-2/3)(-1,0)(-2/3,2/3)};   
   
   \addplot[->] coordinates{
(-0.5,0.5) (-1.2,0.5)}
node[left] {$B_{\mathcal{O}'}$};

   \addplot[] coordinates{
(1,0)}
node[below right] {1};
\addplot[] coordinates{
(-1,0)}
node[below left] {-1};
\addplot[] coordinates{
(0,1)}
node[above right] {1};
\addplot[] coordinates{
(0,-1)}
node[below right] {-1};
   
   \end{axis}
\end{tikzpicture}}
\caption{The unit balls relative to the norm associated to two costs $\cost_\mathcal{O}$ and $\cost_{\mathcal{O}'}$. Here $\mathcal{O}$ is the positive orthant and $\mathcal{O}'$ is its symmetric with respect to the $y$-axis.}
\label{c}
\end{figure}

Let us also denote
$$\tau_\mathcal{O}:= \sign(x_1)(e_1+\dots+e_{N_1}) + \dots + \sign(x_m)(e_{N-N_m+1}+\dots+e_N)\in\R^N,$$
and let $H_\mathcal{O}$ be the unique orthant in $\R^N$ containing the point $\tau_{\mathcal{O}}$. Finally, consider $A_\mathcal{O}:=H_\mathcal{O}\cap B_{\mathcal{O}}$ and {(see Figure \ref{d} and Figure \ref{e})}
$$C_\mathcal{O}:=\{p\in\R^N:\exists q\in A_\mathcal{O} \mbox{ with } (\tau_\mathcal{O})_j(p_j-q_j)\leq 0 \mbox{ for every $j=1,\dots,N$}\}.$$

\begin{figure}[h]
\scalebox{0.85}{
\begin{tikzpicture}
   \begin{axis}[axis x line=middle,axis y line=middle,xmin=-2.4,xmax=2.4,ticks=none,
    ymin=-2,ymax=2,]
\addplot[thick] coordinates{
(-1,1) (1,1) (1,-1) (-1,-1) (-1,1)};
\addplot[pattern = north east lines, opacity=0.2] coordinates{
(-1,1) (1,1) (1,-1) (-1,-1) (-1,1)}; 

\addplot[thick,name path=X,pattern=north west lines] coordinates{
(0,0) (1,0) (1,1) (0,1) (0,0)};
\addplot[thick] coordinates{(1.4,1.5)}
node[] {$H_{\mathcal{O}}$};
\addplot[->] coordinates{(0.5,0.5) (1.3,0.5)}
node[right] {$A_{\mathcal{O}}$};
\addplot[->] coordinates{(0.5,-0.5) (1.3,-0.5)}
node[right] {$B_{\mathcal{O}}$};
\addplot[thick,dashdotted, fill = gray,opacity =0.1] coordinates{(-3,1) (1,1) (1,-3) (-3,-3)};
\addplot[thick,dashdotted] coordinates{(-3,1) (1,1) (1,-3) (-3,-3)};
\addplot[thick, color=gray, opacity=0.8] coordinates{(-1.4,-1.5)}
node[] {$C_{\mathcal{O}}$};

\addplot[] coordinates{
(1,0)}
node[below right] {1};
\addplot[] coordinates{
(-1,0)}
node[below left] {-1};
\addplot[] coordinates{
(0,1)}
node[above right] {1};
\addplot[] coordinates{
(0,-1)}
node[below right] {-1};

   \end{axis} 
\end{tikzpicture}

\begin{tikzpicture}
   \begin{axis}[axis x line=middle,axis y line=middle,xmin=-2.4,xmax=2.4,ticks=none,
    ymin=-2,ymax=2,]
      \addplot[thick] coordinates{
(-2/3,2/3) (0,1) (2/3,2/3) (1,0) (2/3,-2/3)(0,-1)(-2/3,-2/3)(-1,0)(-2/3,2/3)};
\addplot[pattern = north east lines, opacity=0.3] coordinates{
(-2/3,2/3) (0,1) (2/3,2/3) (1,0) (2/3,-2/3)(0,-1)(-2/3,-2/3)(-1,0)(-2/3,2/3)};   
\addplot[thick,name path=X,pattern=north west lines] coordinates{
(0,0) (0,1) (-2/3,2/3) (-1,0) (0,0)};
\addplot[thick] coordinates{(-1.4,1.5)}
node[] {$H_{\mathcal{O}'}$};
\addplot[->] coordinates{(-0.5,0.5) (-1.3,0.5)}
node[left] {$A_{\mathcal{O}'}$};
\addplot[->] coordinates{(-0.5,-0.5) (-1.3,-0.5)}
node[left] {$B_{\mathcal{O}'}$};
\addplot[thick,dashdotted, fill = gray,opacity =0.1] coordinates{(3,1) (0,1)(-2/3,2/3) (-1,0) (-1,-3) (3,-3)};
\addplot[thick,dashdotted] coordinates{(3,1) (0,1)(-2/3,2/3) (-1,0) (-1,-3) (3,-3)};
\addplot[thick, color=gray, opacity=0.8] coordinates{(1.4,-1.5)}
node[] {$C_{\mathcal{O}'}$};

\addplot[] coordinates{
(1,0)}
node[below right] {1};
\addplot[] coordinates{
(-1,0)}
node[below left] {-1};
\addplot[] coordinates{
(0,1)}
node[above right] {1};
\addplot[] coordinates{
(0,-1)}
node[below right] {-1};
   \end{axis} 
\end{tikzpicture}}
\caption{The construction of the sets $C_\mathcal{O}$ and $C_{\mathcal{O}'}$.}
\label{d}
\end{figure}
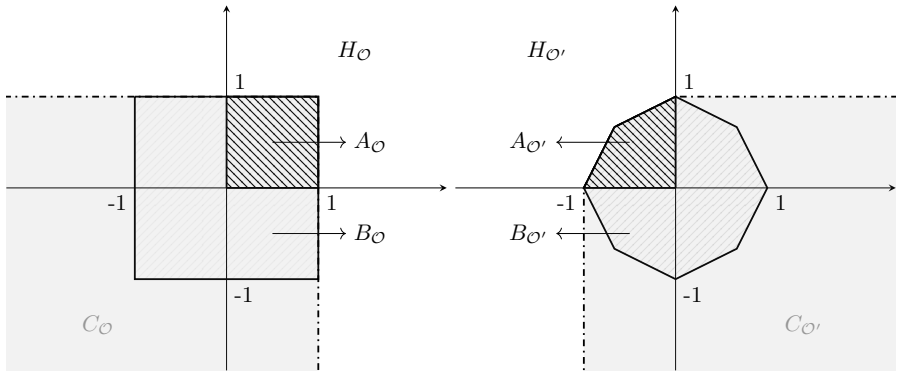

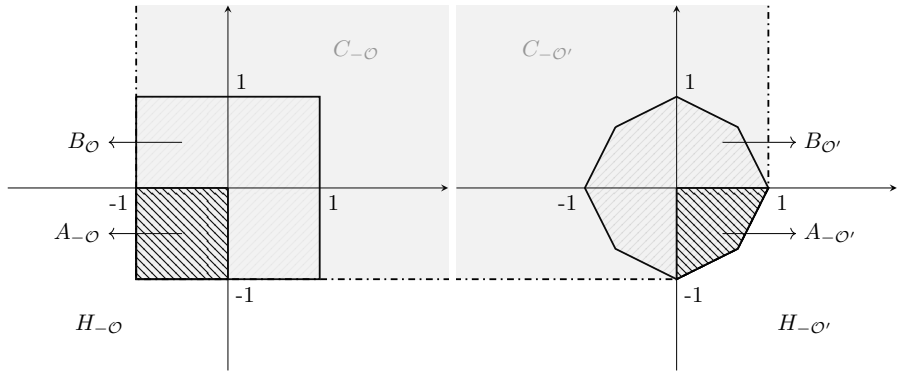
\begin{figure}
\scalebox{0.85}{
\begin{tikzpicture}
   \begin{axis}[axis x line=middle,axis y line=middle,xmin=-2.4,xmax=2.4,ticks=none,
    ymin=-2,ymax=2,]
\addplot[thick] coordinates{
(-1,1) (1,1) (1,-1) (-1,-1) (-1,1)};
\addplot[pattern = north east lines, opacity=0.2] coordinates{
(-1,1) (1,1) (1,-1) (-1,-1) (-1,1)}; 

\addplot[thick,name path=X,pattern=north west lines] coordinates{
(0,0) (-1,0) (-1,-1) (0,-1) (0,0)};
\addplot[thick] coordinates{(-1.4,-1.5)}
node[] {$H_{-\mathcal{O}}$};
\addplot[->] coordinates{(-0.5,-0.5) (-1.3,-0.5)}
node[left] {$A_{-\mathcal{O}}$};
\addplot[->] coordinates{(-0.5,0.5) (-1.3,0.5)}
node[left] {$B_{\mathcal{O}}$};
\addplot[thick,dashdotted, fill = gray,opacity =0.1] coordinates{(3,-1) (-1,-1) (-1,3) (3,3)};
\addplot[thick,dashdotted] coordinates{(3,-1) (-1,-1) (-1,3) (3,3)};
\addplot[thick, color=gray, opacity=0.8] coordinates{(1.4,1.5)}
node[] {$C_{-\mathcal{O}}$};

\addplot[] coordinates{
(1,0)}
node[below right] {1};
\addplot[] coordinates{
(-1,0)}
node[below left] {-1};
\addplot[] coordinates{
(0,1)}
node[above right] {1};
\addplot[] coordinates{
(0,-1)}
node[below right] {-1};
   \end{axis} 
\end{tikzpicture}

\begin{tikzpicture}
   \begin{axis}[axis x line=middle,axis y line=middle,xmin=-2.4,xmax=2.4,ticks=none,
    ymin=-2,ymax=2,]
      \addplot[thick] coordinates{
(-2/3,2/3) (0,1) (2/3,2/3) (1,0) (2/3,-2/3)(0,-1)(-2/3,-2/3)(-1,0)(-2/3,2/3)};
\addplot[pattern = north east lines, opacity=0.3] coordinates{
(-2/3,2/3) (0,1) (2/3,2/3) (1,0) (2/3,-2/3)(0,-1)(-2/3,-2/3)(-1,0)(-2/3,2/3)};   
\addplot[thick,name path=X,pattern=north west lines] coordinates{
(0,0) (0,-1) (2/3,-2/3) (1,0) (0,0)};
\addplot[thick] coordinates{(1.4,-1.5)}
node[] {$H_{-\mathcal{O}'}$};
\addplot[->] coordinates{(0.5,-0.5) (1.3,-0.5)}
node[right] {$A_{-\mathcal{O}'}$};

\addplot[->] coordinates{(0.5,0.5) (1.3,0.5)}
node[right] {$B_{\mathcal{O}'}$};

\addplot[thick,dashdotted, fill = gray,opacity =0.1] coordinates{(-3,-1) (0,-1)(2/3,-2/3) (1,0) (1,3) (-3,3)};
\addplot[thick,dashdotted] coordinates{(-3,-1) (0,-1)(2/3,-2/3) (1,0) (1,3) (-3,3)};
\addplot[thick, color=gray, opacity=0.8] coordinates{(-1.4,1.5)}
node[] {$C_{-\mathcal{O}'}$};

\addplot[] coordinates{
(1,0)}
node[below right] {1};
\addplot[] coordinates{
(-1,0)}
node[below left] {-1};
\addplot[] coordinates{
(0,1)}
node[above right] {1};
\addplot[] coordinates{
(0,-1)}
node[below right] {-1};
   \end{axis} 

\end{tikzpicture}
}
\caption{{The construction of the sets $C_\mathcal{-O}$ and $C_{\mathcal{-O}'}$.}}
\label{e}
\end{figure}

Observe that 
\begin{equation}\label{eqn_observ}
C_\mathcal{O}\cap H_\mathcal{O}=A_\mathcal{O},
\end{equation}
by the monotonicity of $A_\mathcal{O}$, which is implied by the monotonicity of $B_\mathcal{O}$ (the intersection of monotone sets is monotone). Lastly we denote $$C:=\bigcap_{\mathcal{O}\subset\R^m}C_\mathcal{O},$$
where the intersection is taken among the $2^m$ orthants in $\R^m$ {(see Figure \ref{f})}.

\begin{figure}[h]
\scalebox{0.85}{
\begin{tikzpicture}
   \begin{axis}[axis x line=middle,axis y line=middle,xmin=-2.4,xmax=2.4,ticks=none,
    ymin=-2,ymax=2,]
      \addplot[thick,name path=X,fill = gray,opacity =0.2] coordinates{
(-1,0) (-2/3,2/3)(0,1) (1,1) (1,0) (2/3,-2/3) (0,-1) (-1,-1) (-1,0)};
\addplot[] coordinates{(0.5,0.5)}
   node[] {$C$};
 
   \addplot[name path=X] coordinates{
(-1,0) (-2/3,2/3)(0,1) (1,1) (1,0) (2/3,-2/3) (0,-1) (-1,-1) (-1,0)};

\addplot[] coordinates{
(1,0)}
node[below right] {1};
\addplot[] coordinates{
(-1,0)}
node[below left] {-1};
\addplot[] coordinates{
(0,1)}
node[above right] {1};
\addplot[] coordinates{
(0,-1)}
node[below right] {-1};
   \end{axis} 
\end{tikzpicture}}
\caption{The construction of the set $C$.}
\label{f}
\end{figure}
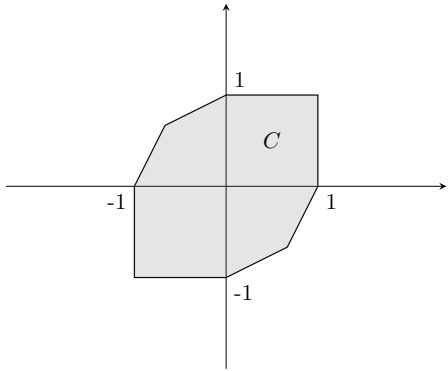

We claim that $C$ is a closed, convex, and monotone set, with non-empty interior, which is symmetric with respect to the origin, bounded, and satisfies
\begin{equation}\label{eqn_norma_sbilenca}
C\cap H_\mathcal{O}=A_\mathcal{O},\quad\mbox{ for every orthant $\mathcal{O}\subset\R^m$}.
\end{equation}
The monotonicity of $C$ would imply that the norm $\|\cdot\|$ on $\R^N$, whose unit ball is $C$, is monotone. Moreover, Step 1 implies that \eqref{eqn_main} holds for the norm $\|\cdot\|_\mathcal{O}$ and the cost $\mathcal{C}_\mathcal{O}$. We also observe that, in the orthant $\mathcal{O}$ of $\R^m$, the costs $\mathcal{C}$ and  $\mathcal{C}_\mathcal{O}$ coincide. Also, \eqref{eqn_norma_sbilenca} implies that, in the orthant $H_\mathcal{O}$ of $\R^N$, $\|\cdot\|$ and $\|\cdot\|_\mathcal{O}$ coincide. In this way, we get that \eqref{eqn_main} holds also for $\|\cdot\|$ and for the cost $\cost$. 

The fact that $C$ is closed, convex and monotone follows from the fact that each set $C_\mathcal{O}$ is so, moreover each $C_\mathcal{O}$ contains a neighborhood of the origin, hence $C$ has non-empty interior. The fact that $C$ is bounded and symmetric with respect to the origin follows from the fact that $C_\mathcal{O}\cap C_\mathcal{-O}$ is so for every $\mathcal{O}$, where we denoted by $-\mathcal{O}$ the orthant which is symmetric to $\mathcal{O}$ with respect to the origin. To conclude, we have to prove \eqref{eqn_norma_sbilenca}. To prove \eqref{eqn_norma_sbilenca}, we make the following claim:
\begin{equation}\label{C1}
\tag{\text{Claim} 1}
A_{\mathcal{O}'}\cap H_{\mathcal{O}}=A_{\mathcal{O}}\cap H_{\mathcal{O}'},\quad\mbox{ for every pair of orthants $\mathcal{O}, \mathcal{O}'\subset\R^m$}.
\end{equation}

\vspace{0.2cm}
\fbox{\it Proof of \eqref{eqn_norma_sbilenca} using \eqref{C1}}
\vspace{0.2cm}

Let us show firstly how \eqref{C1} implies \eqref{eqn_norma_sbilenca}. By the definition of $C$, it is sufficient to show that
\begin{equation}\label{eqn_norma_sbilenca_2}
C_{\mathcal{O}'}\cap A_{\mathcal{O}}=A_{\mathcal{O}},\quad\mbox{ for every pair of orthants $\mathcal{O}, \mathcal{O}'\subset\R^m$},
\end{equation}
indeed
$$C\cap H_\mathcal{O}= \bigcap_{\mathcal{O'}\subset\R^m}C_\mathcal{O'}\cap H_\mathcal{O}=\bigcap_{\mathcal{O'}\subset\R^m}C_\mathcal{O'}\cap C_\mathcal{O}\cap H_\mathcal{O}\stackrel{\eqref{eqn_observ}}{=}\bigcap_{\mathcal{O'}\subset\R^m}C_\mathcal{O'}\cap A_\mathcal{O}\stackrel{\eqref{eqn_norma_sbilenca_2}}{=}A_\mathcal{O}.$$
To prove \eqref{eqn_norma_sbilenca_2} using \eqref{C1}, we write
\[
C_{\mathcal{O}'}\cap A_{\mathcal{O}} = \bigcup_{H}(C_{\mathcal{O}'}\cap H\cap A_{\mathcal{O}}),
\]
where $H$ varies among the $2^N$ orthants of $\R^N$. Then \eqref{eqn_norma_sbilenca_2} would follow from
\begin{equation}\label{eqn_rott}
C_{\mathcal{O}'}\supseteq H\cap A_{\mathcal{O}},\quad  \forall \mathcal{O},\mathcal{O}',H.
\end{equation}
To prove \eqref{eqn_rott}, consider $z \in H\cap A_{\mathcal{O}}$. We define a new vector, $y\in \mathbb{R}^N$, in this way:
\[
y_j =
\begin{cases}
z_j, \text{ if $(\tau_{\mathcal{O}'})_jz_j \geq 0$, }\\
0, \text{ otherwise}.
\end{cases}
\]
It is immediate to see that $(\tau_{\mathcal{O}'})_j(z_j - y_j) \leq 0, \forall j$, and that $y\in H_{\mathcal{O}'}$. Hence, to prove that $z\in C_{\mathcal{O}'}$ it is sufficient to prove that $y \in A_{\mathcal{O}'}$. We observe that $y \in A_\mathcal{O}$, because $y\preceq z$ and $z\in A_\mathcal{O}$. By \eqref{C1}, this yields: $$y \in A_\mathcal{O}\cap H_{\mathcal{O}'}= A_{\mathcal{O}'}\cap H_{\mathcal{O}}.$$ Therefore $y \in A_{\mathcal{O}'}$ as desired.\\ 

\vspace{0.2cm}
\fbox{\it Proof of \eqref{C1}: strategy}
\vspace{0.2cm}

To prove \eqref{C1}, we will prove the more precise formula:
\begin{equation}\label{eqn_interfaces}
A_{\mathcal{O}'}\cap H_{\mathcal{O}} = {\co}(\{0\}\cup\{q_D:D\in (\bar{\mathcal{A}}\cap H_\mathcal{O}\cap H_{\mathcal{O}'}\setminus\{0\})\}),\quad\forall \mathcal{O},\mathcal{O}' \subset \mathbb{R}^m.
\end{equation}
Equation \eqref{eqn_interfaces} implies \eqref{C1} because its right hand side does not change if we swap $\mathcal{O}$ and $\mathcal{O}'$.
Denote $E$ the RHS of \eqref{eqn_interfaces}, i.e. 
$$E:={\co}(\{0\}\cup\{q_D:D\in (\bar{\mathcal{A}}\cap H_\mathcal{O}\cap H_{\mathcal{O}'}\setminus\{0\})\}),$$ and observe that $E\subseteq A_{\mathcal{O}'}\cap H_{\mathcal{O}}$, since $q_D \in A_{\mathcal{O}'}\cap H_{\mathcal{O}}$ for every $D \in \bar{\mathcal{A}}\cap H_\mathcal{O}\cap H_{\mathcal{O}'}\setminus\{0\}$, by Step 1.

In order to prove \eqref{eqn_interfaces}, we need to prove the reverse inclusion, hence, since $A_{\mathcal{O}'}\cap H_{\mathcal{O}}$ is a compact convex set, we can assume by contradiction that (by Krein-Milman Theorem) there exists an extreme point $z$ of $A_{\mathcal{O}'}\cap H_{\mathcal{O}}$ that does not belong to $E$.

Since $z\in A_{\mathcal{O}'}\subset B_{\mathcal{O}'}$, we can write it as a convex combination
$$z=\sum_{k=1}^K\lambda_k q_{D^k},$$
where we remind that $q_{D^k}=\frac{D^k}{c_{D^k}}$, $D^k\in\bar{\mathcal{A}}$ and $c_{D^k}:=c_{\bar D^k,0}$ are those defined in \eqref{eqn_defcab}, with $\cost_{\mathcal{O}'}$ in place of $\cost$.
Our aim is to replace the elements $q_{D^k}$ appearing in the convex combination above with suitable elements $q_{F^k}$, where the points $F^k$ belong to $\bar{\mathcal{A}}\cap H_\mathcal{O}\cap H_{\mathcal{O}'}$. Firstly we will prove only that one can write $z$ as a convex combination of $0$ and some points $q_{G^k}$, where the points $G_k$ can be chosen in $\bar{\mathcal{A}}\cap {\rm{span}}\{H_\mathcal{O}\cap H_{\mathcal{O}'}\}$. Then we will reduce to the points $q_{F^k}$ with $F^k$ in $\bar{\mathcal{A}}\cap H_\mathcal{O}\cap H_{\mathcal{O}'}$ which would give a contradiction to the fact that $z\not\in E$.

\vspace{0.2cm}
\fbox{\it Proof of \eqref{C1}: first reduction}
\vspace{0.2cm}

For $k=1,\dots,K$, we define vectors $G^k:=(g^k_1,\dots,g^k_N)\in \bar{\mathcal{A}}$ by
\begin{equation}\label{zerozero}
\begin{cases}
g^k_j := 0, \text{ if $z_j$ = 0}\\
g^k_j := d^k_j, \text{ otherwise. }
\end{cases}
\end{equation}
Note that, as a consequence of \eqref{zerozero} and the fact that $z\in H_{\mathcal{O}'}\cap H_\mathcal{O}$, for every $k$ we have that $G^k\in \bar{\mathcal{A}}\cap{\rm span}\{H_\mathcal{O}\cap H_{\mathcal{O}'}\}$. Moreover
\begin{equation}\label{eqn_zet}
z=\sum_{k=1}^K\lambda_k\frac{G^k}{c_{D^k}}.
\end{equation}
We also have that
\begin{equation}\label{monzet}
c_{G^k} \leq c_{D^k}
\end{equation}
by the monotonicity of the cost $\cost_{\mathcal{O}'}$. Hence we can write, denoting $\lambda'_k:=\lambda_k\frac{c_{G^k}}{c_{D^k}}\leq\lambda_k$,
$$z=\sum_{k=1}^K\lambda'_kq_{G^k}.$$
Hence we have written $z$ as a convex combination of $0$ and some $q_{G^k}$, for points 
$$G^k\in\bar{\mathcal{A}}\cap {\rm{span}}\{H_\mathcal{O}\cap H_{\mathcal{O}'}\}.$$

\vspace{0.2cm}
\fbox{\it Proof of \eqref{C1}: second reduction}
\vspace{0.2cm}

Let now $\eta:=(\eta_1,\dots,\eta_N)$ be a point in the relative interior of $H_\mathcal{O}\cap H_{\mathcal{O}'}$.
Note that one can choose $\eta = \tau_{\mathcal{O}} + \tau_{\mathcal{O}'}$. Indeed, the relative interior of $H_{\mathcal{O}}\cap H_{\mathcal{O}'}$ is the set of points $p = (p_1.\dots,p_N)\in \R^N$ such that
\[
\begin{cases}
p_j = 0, \text{ if } (\tau_{\mathcal{O}})_j(\tau_{\mathcal{O}'})_j = - 1,\\
 \sign(p_j) = (\tau_{\mathcal{O}})_j, \text{ if } (\tau_{\mathcal{O}})_j(\tau_{\mathcal{O}'})_j = 1.
\end{cases}
\]
For $k=1,\dots,K$, we define vectors $F^k:=(f^k_1,\dots,f^k_N)\in \bar{\mathcal{A}}$ by
\begin{equation}\label{zero2}
\begin{cases}
f^k_j := g^k_j, \text{ if $\eta_jg^k_j \geq 0$}\\
f^k_j := -g^k_j, \text{ otherwise. }
\end{cases}
\end{equation}
Since, for every $k$, $G^k\in \bar{\mathcal{A}}\cap{\rm span}\{H_\mathcal{O}\cap H_{\mathcal{O}'}\}$, then $F^k\in \bar{\mathcal{A}}\cap H_\mathcal{O}\cap H_{\mathcal{O}'}$. Indeed $H_{\mathcal{O}}\cap H_{\mathcal{O}'}$ is the set of points $p = (p_1,\dots,p_N)\in {\rm span}\{H_\mathcal{O}\cap H_{\mathcal{O}'}\}\subset \R^N$ such that $p_j\eta_j \ge 0$ for every $j$. Since $z\in H_\mathcal{O}\cap H_{\mathcal{O}'}$ and, since $\cost_{\mathcal{O}'}$ satisfies \eqref{supersymm}, it holds $c_{G^k}=c_{F^k}$ for every $k$, then
\begin{equation}\label{eqn_zet2}
z=\sum_{k=1}^K\lambda'_kq_{G^k}\preceq \sum_{k=1}^K\lambda'_k\frac{F^k}{c_{G^k}}=\sum_{k=1}^K\lambda'_kq_{F^k}=:z',
\end{equation}
where we remind that $\preceq$ is the order relation defined in \eqref{order}. 

\vspace{0.2cm}
\fbox{\it Proof of \eqref{C1}: last contradiction, i.e. $z=z'$}
\vspace{0.2cm}

We observe that by \eqref{eqn_zet2} and since $F^k\in \bar{\mathcal{A}}\cap H_\mathcal{O}\cap H_{\mathcal{O}'}$ it follows that $z'\in E$.
We will prove now that $z'=z$, which would be a contradiction, since $z\not\in E$ by assumption. Assume by contradiction that $z'\neq z$ and observe that 
\begin{equation}\label{annul}
z_j = 0 \text{ for some } j \Rightarrow z_j' = 0.
\end{equation}
Indeed, $\eqref{zerozero}$ yields that $g_j^k = 0$ for every $k$, if $z_j = 0$, and then by $\eqref{zero2}$ also $f_j^k = 0$. Therefore, by \eqref{eqn_zet2}, $$z'_j = \sum_{k=1}^K\lambda'_k\frac{f_j^k}{c_{F^k}} = 0.$$
Define now, for $\varepsilon>0$, $w_\varepsilon:=z-\varepsilon(z'-z)$. Note that, by \eqref{annul}, we have $w_\varepsilon\preceq z\preceq z'$ for $\varepsilon$ sufficiently small. This implies that, for $\varepsilon$ sufficiently small, $w_\varepsilon\in A_{\mathcal{O}'}\cap H_{\mathcal{O}}$, because $A_{\mathcal{O}'}\cap H_{\mathcal{O}}$ is an intersection of monotone sets, therefore monotone. Hence we can write $z$ as a non-trivial convex combination of the points $w_\varepsilon$ and $z'$ in $A_{\mathcal{O}'}\cap H_{\mathcal{O}}$, which violates the extremality of $z$.  
\end{proof}

As we observed before, Theorem \ref{main} provides a proof of the existence of a solution to the MMTP, which does not require a proof of the lower semicontinuity of the energy $\mathbb{E}$.
\begin{cor} Under assumptions {\rm (i),(ii),${\rm (iii')}$} on the cost functional $\cost$, the problems {\rm MMTP} and {\rm MMP} admit a solution. 
\end{cor}
\begin{proof}
The fact that the MMP admits a solution follows from Theorem \ref{thm:cptness}. The fact that the MMTP admits a solution then follows from Theorem \ref{main}.
\end{proof}

The property ${\rm (iii')}$ of Definition \ref{def:multi-material_cost} appears to be the most restrictive. However, at least in the ``single-material'' case is also necessary to obtain the equivalence with the mass-minimization problem.

\begin{thm}
If $m = 1$, and $\mathcal{C}$ is a cost that fulfills {\rm (i),(ii)} of Definition \ref{def:multi-material_cost}, then $\rm{(iii')}$ holds if and only if there exists a monotone norm $\|\cdot\|$ that satisfies \eqref{eqn_main}.
\end{thm}
\begin{proof}
One implication has already been proven in Theorem \ref{thm:norm}. Suppose now that there exists a monotone norm $\|\cdot\|$ on $\mathbb{R}^N$ that satisfies \eqref{eqn_main}. Fix any $E \in {\mathcal{A}}$, where $\mathcal{A}$ is defined at the beginning of the proof of Theorem \ref{thm:norm}. We can write $E$ as
\[
E = \sum_{k \in K} {\bf e}_{i_k},
\]
being $K$ a subset of $\{1,\dots,N\}$. We denote with $\#K$ the cardinality of $K$. By \eqref{eqn_main}, we have:
\begin{equation}\label{eq:thm}
\|E\| = \mathcal{C}(\#K) = \|F\|,
\end{equation}
for any $F\in{\mathcal{A}}$ such that $F = \sum_{k \in K'}{\bf e}_{i_k}$ and $\#K' = \#K$. For every $\ell\in K$ define $K_{\ell} := K\setminus \{\ell\}$.
Define $E_\ell := \sum_{k \in K_\ell}{\bf e}_{i_k}$.
Therefore,
\[
(\#K - 1)E = \sum_{\ell \in K} E_\ell,
\]
and, by \eqref{eq:thm}, we get
\[
(\#K - 1)\mathcal{C}(\#K)= (\#K - 1)\|E\| = \big\|\sum_{\ell \in K} E_\ell\big\| \le \sum_{\ell \in K}\| E_\ell\|= \sum_{\ell \in K} \mathcal{C}(\#K - 1) = \#K\mathcal{C}(\#K - 1).
\]

Since $K\subset\{1,\dots,N\}$ is arbitrary, we obtain, $\forall x \in \{2,\dots,N\}$,

\[
\frac{C(x)}{x} \le \frac{C(x - 1)}{x - 1},
\]
and, by induction,
\[
\frac{C(x)}{x} \le \frac{C(y)}{y}, \text{ if } 1\leq y \leq x.
\]
\end{proof}
It is well-known that, if $\mathcal{C}: [0,\infty)\to[0,\infty)$ is concave and $\cost(0)=0$, then the quantity $\frac{\mathcal{C}(x)}{x}$ is non-increasing. Hence we obtain the following corollary.
\begin{cor}
In the case $m = 1$, Theorem \ref{thm:norm} holds if ${\rm (iii')}$ is replaced by the request that $\mathcal{C}$ coincide on $\N$ with a concave function.
\end{cor}
\begin{remark}
{\rm The previous corollary allows us to include in the list of cost functionals for which Theorem \ref{main} applies the cost considered in \cite{Bran_WirthUP}, which describes a model for the urban planning (or a discrete version of it, in our case). More precisely the cost is $\cost(z)=\min\{az;z+b\}$ with $a>1, b>0$, which is clearly concave.}
\end{remark}

\section{Properties of minimizers}
Most of the regularity properties for classical continuous models of single-material branched transportation, such as single-path properties and finite tree structure away from the boundary (see \cite{Bernot2009}) are deduced using a crucial property of discrete optimal networks, which is the absence of cycles. Even in our case, removing cycles from each of the $m$ components of a competitor for the MMTP does not increase the energy (note that we have used this fact in the proof of Theorem \ref{main}). Nevertheless it might happen that the operation does not \emph{strictly} reduce the energy as well, and in particular minimizers could contain cycles. The aim of this section is to provide a simple example of such phenomenon. 

Consider the multi-material cost
\[
\cost(\theta_1,\theta_2,\theta_3,\theta_4) = \max\{|\theta_1| + |\theta_2| + |\theta_3|, |\theta_4|\}.
\]
Since the cost is additive in the first three variables, it follows that, for every boundary datum $\mathcal B$ whose fourth component is trivial, a solution to the associated MMTP can be obtained as a superposition of the solutions of three single-material problems (see Remark \ref{rem:sum}). Namely those minimization problems whose boundaries are defined respectively by the three components $\mathcal B_1$, $\mathcal B_2$, and $\mathcal B_3$ of $\mathcal B$ (and the corresponding single-material cost is simply $\cost(\theta)=|\theta|$ for $\theta\in\Z$).

Let us now fix a specific boundary $\mathcal{B}$.
Take three non-collinear points $x_1$, $x_2$, and $x_3$ on $\R^2$ and denote
$$\mathcal{B}:=(-1,0,1,0)\delta_{x_1}+(1,-1,0,0)\delta_{x_2}+(0,1,-1,0)\delta_{x_3}.$$
By the discussion above, a minimizer for the MMTP associated to the cost $\cost$ for the boundary $\mathcal{B}$ is the 1-dimensional integral $\Z^4$-current $T$, which is written in component as $T:=(T_1,T_2,T_3,0)$, where
\begin{itemize}
\item[(i)] $T_1:=\llbracket\overline{x_1x_2},\tau_1,1\rrbracket$ is the classical integral current associated to the segment $\overline{x_1,x_2}$ oriented from $x_1$ to $x_2$ with unit multiplicity; 
\item[(ii)] $T_2:=\llbracket\overline{x_2x_3},\tau_2,1\rrbracket$;
\item[(iii)] $T_3:=\llbracket\overline{x_3x_1},\tau_3,1\rrbracket$.
\end{itemize}
Observe also that the 1-dimensional integral $\Z^4$-current $T':=(T_1,T_2,T_3,T_1+T_2+T_3)$ satisfies $\en(T')=\en(T)$, and since $\partial(T_1+T_2+T_3)=0$, it follows that $\partial T'=\partial T$, hence $T'$ is also a minimizer of the MMTP associated to $\mathcal{B}$. Observe that not only $T'$ contains a topological cycle in its support (that property holds for $T$ itself), but the fourth components of $T'$ contains a cycle (actually it \emph{is} a cycle) in the sense of currents.

One could find this example unsatisfactory, because the material associated to the cyclic component of $T'$ does not appear in the boundary datum. Nevertheless it is easy to modify the example above in order to add the fourth material in the boundary datum, still obtaining the previous phenomenon. More precisely, denoting $T_4$ a non-trivial oriented segment which is ``very far'' from the supports of $T_1,T_2$, and $T_3$, then clearly the current $T'':=T'+(0,0,0,T_4)$ is also a minimizer for the corresponding boundary: roughly speaking, even if { in general it would be convenient for the fourth material to interact with the first three, there is no convenience in this case, due to the large distance of the corresponding sources and sinks} {(see Figure \ref{g})}.  
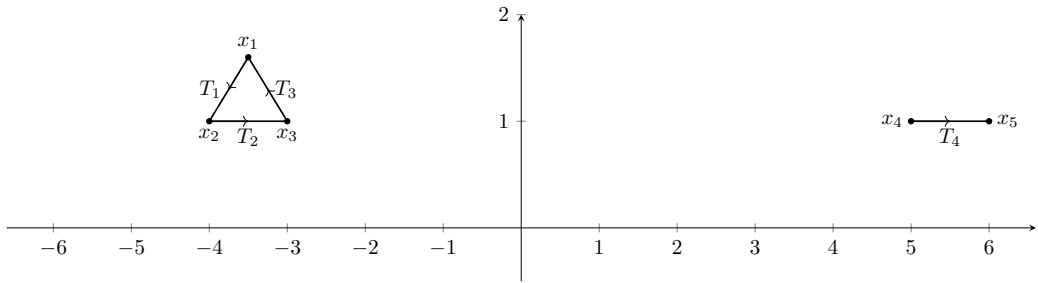
\begin{figure}[h]
\scalebox{0.8}{
\begin{tikzpicture}
   \begin{axis}[axis x line=middle,axis y line=middle,xmin=-6,xmax=6,enlarge x limits = 0.05, ymin=-0.5,ymax=2, height = 6 cm, width = 1.2 \linewidth]
     
\addplot[thick] coordinates{
(-4,1) (-3,1) (-3.5,1.6) (-4,1) };

\addplot[->] coordinates{
(-3.5,1.6) (-3.75,1.3) }
node[left]{$T_1$};

\addplot[->] coordinates{
(-4,1) (-3.5,1) }
node[below]{$T_2$};

\addplot[->] coordinates{
(-3,1) (-3.25,1.3) }
node[right]{$T_3$};

\addplot[->] coordinates{
(5,1) (5.5,1) }
node[below]{$T_4$};

\addplot[thick] coordinates{
(-4,1)}
node[below]{$x_2$};
\addplot[thick] coordinates{
(-4,1)}
node[circle,fill=black,inner sep=0pt,minimum size=3pt]{};
\addplot[thick] coordinates{
(-3,1)}
node[below]{$x_3$};
\addplot[thick] coordinates{
(-3,1)}
node[circle,fill=black,inner sep=0pt,minimum size=3pt]{};
\addplot[thick] coordinates{
(-3.5,1.6)}
node[above]{$x_1$};
\addplot[thick] coordinates{
(-3.5,1.6)}
node[circle,fill=black,inner sep=0pt,minimum size=3pt]{};

\addplot[thick] coordinates{
(5,1) (6,1)};
\addplot[thick] coordinates{
(5,1)}
node[left]{$x_4$};
\addplot[thick] coordinates{
(5,1)}
node[circle,fill=black,inner sep=0pt,minimum size=3pt]{};
\addplot[thick] coordinates{
(6,1)}
node[right]{$x_5$};
\addplot[thick] coordinates{
(6,1)}
node[circle,fill=black,inner sep=0pt,minimum size=3pt]{};

\end{axis}
   
\end{tikzpicture}
}
\caption{{The component $T_4$ does not interact with $T_1, T_2$, and $T_3$ because the corresponding boundaries are too far away.}}
\label{g}
\end{figure}

As it was observed in \cite{Bran_Wirth}, in order to get better properties of minimizers of single-material branched transportation problem it is necessary to require the concavity of the cost. In our case we will require that the cost is concave in every component. In this case it is possible to prove that there exists a solution $T$ of the MMTP whose components $T_i$ $(i=1,\dots,m)$ are all supported on trees (in particular they are acyclic currents). Let us stress that this does not imply the absence of loops in the support of $T$, but only in the support of each $T_i$. More precisely, we have the following two propositions, which are the analogue of \cite[Lemma 2.6, Remark 2.7]{Bran_Wirth}. The proofs are also analogous. We say that a multi-material cost $\cost:\Z^m\to\R$ is \emph{concave} {(resp. \emph{strictly concave})} if it coincides on $\Z^m$ with a function $f:\R^m\to\R$ which is concave {(resp. strictly concave)} in every component, i.e. $f(z_1,\dots,z_m)$ is a concave (resp. strictly concave) function of each variable $z_i$. We call a \emph{tree} a set in $\R^d$ which does not contain the support of any non-trivial closed curve.
\begin{prop}
Let $\cost$ be a concave multi-material cost. Then for every 1-dimensional integral $\Z^m$-current $T$, there exists another current $T'$ with $\partial T'=\partial T$, $\en(T')\leq\en(T)$, and $T'_i$ is supported on a tree for every $i=1,\dots,m$. 
\end{prop}

{
\begin{prop}
Let $\cost$ be a strictly concave multi-material cost. Then for every 1-dimensional integral $\Z^m$-current $T$ which is a solution to the MMTP associated to its boundary, every component $T_i$ of $T$ is supported on a tree. 
\end{prop}
}

\section{Examples}

In this section, we consider some concrete costs functional $\cost$ and we exhibit a possible norm $\|\cdot\|$ which turns a MMTP associated to such cost into a mass-minimization problem. {At the end of the section we also provide some examples of calibrations.} 

\subsection{Examples of costs}
\begin{itemize}
\item[(1)] \emph{Steiner energy.} For $m=1$, let 
\[
\cost(z):=
\begin{cases}
0, & z = 0,\\
1, & z \neq 0.
\end{cases}
\]
The minimization of the energy $\en$ associated to such cost corresponds to the minimization of the size functional. Clearly the corresponding norm $\|\cdot\|$ on $\R^N$ given by Theorem \ref{main} is simply the supremum norm.

\item[(2)] \emph{Gilbert-Steiner energy.} For $m=1$, fix $0\leq\alpha\leq1$ and let
\[
\cost(z):=
\begin{cases}
0, & z = 0,\\
|z|^\alpha, & z \neq 0.
\end{cases}
\]
The minimization of the corresponding energy $\en$ corresponds to the minimization of the $\alpha$-mass (see e.g. \cite{Xia}). As it is shown in \cite{AnMa2}, the corresponding norm $\|\cdot\|$ on $\R^N$ is the $p$-norm with $p=\frac{1}{\alpha}$.  Note that for $\alpha=0$ we recover the Steiner energy.

\item[(3)] \emph{Linear combinations.} For $m=1$, fix $K\in\N$ and for $k=1,\dots, K$ let $0\leq\alpha_k\leq1$ and let $\lambda_k>0$. Define
\[
\cost(z):=
\begin{cases}
0, & z = 0,\\
\sum_{k=1}^K\lambda_k|z|^{\alpha_k}, & z \neq 0.
\end{cases}
\]
It is easy to see that $\cost$ satisfies properties (i),(ii), and ${\rm (iii')}$ of Definition \ref{def:multi-material_cost}. The corresponding norm $\|\cdot\|$ on $\R^N$ is $\|x\|=\sum_{k=1}^K\lambda_k|x|^{p_k}$, where $p_k=\frac{1}{\alpha_k}$. Such a cost is considered for example in \cite{Cham_Mer_Fer} in order to approximate the Steiner energy and to perform numerical simulations.

\item[(4)] \emph{Supremum of costs.} For $m=1$, fix $K\in\N$ and for $k=1,\dots, K$ let $\cost_k$ be a cost functional satisfying properties (i),(ii), and ${\rm (iii')}$ of Definition \ref{def:multi-material_cost}. Define
\[
\cost(z):=\max_{k=1,\dots,K}\cost_k(x).
\]
The corresponding norm $\|\cdot\|$ on $\R^N$ is the maximum of the norms associated to each $\cost_k$. 
\item[(5)] \emph{PLC technology.} For $m=2$, let $0<\alpha_1\ll\alpha_2\leq 1$. Define
\[
\cost(z_1,z_2):=\max\{\lambda_1|z_1|^{\alpha_1};\lambda_2|z_2|^{\alpha_2}\},
\]
with $\lambda_1,\lambda_2>0$. A monotone norm $\|\cdot\|$ on $\R^N$ which satisfies \eqref{eqn_main} is 
$$\|(x_1,\dots,x_{N_1},y_1,\dots,y_{N_2})\|=\max\{\lambda_1|(x_1,\dots,x_{N_1})|_{p_1};\lambda_2|(y_1,\dots,y_{N_2})|_{p_2}\}.$$
where $p_i=\alpha_i^{-1}$ for $i=1,2$. The fact that $\alpha_1\ll\alpha_2$ express the idea that once the infrastructure transporting the second material (i.e., the electricity) is built one can add ``almost any'' quantity of the first material (i.e., Internet signal) for free.
\item[(6)] \emph{Composite multi-material costs.} For general $m\geq 2$, consider any monotone norm $|\cdot|_{\star}$ in $\R^m$ and single-material costs $\cost_1,\dots,\cost_m: \Z\to \R$, associated to monotone norms $\|\cdot\|^1, \dots, \|\cdot\|^{m}$ on $\R^{N_1},\dots,\R^{N_m}$, respectively. Define
\[
\cost(z_1,\dots,z_m):= |(\mathcal{C}_1(z_1),\dots,\mathcal{C}_m(z_m))|_\star.
\]
A monotone norm $\|\cdot\|$ on $\R^N$ which satisfies \eqref{eqn_main} for the multi-material cost $\cost$ is 
$$\|(x_1,\dots,x_N)\|= |(\|(x_1,\dots,x_{N_1})\|^{1},\dots,\|(x_{N-N_m+1},\dots,x_{N})\|^{m})|_\star.$$
Observe that the cost associated to the PLC technology corresponds to the choice $|\cdot|_\star =\|\cdot\|_\infty$ on $\R^2$, $\cost_1(z)=\lambda_1|z|^{\alpha_1}$, and $\cost_2(z)=\lambda_2|z|^{\alpha_2}$.

%\item{} \emph{Composite multi-material costs.} For general $r\in \mathbb{N}$ and $m_k\in \mathbb{N}$, for $k \in \{1,\dots,r\}$, define
%\[
%\cost((x^1_1,\dots,x^1_{m_1},\dots,x^r_1,\dots,x^r_{m_r})):= \|(\mathcal{C}_1(x_1^1,\dots,x^1_{m_1}),\dots,\mathcal{C}_r(x^r_{1},\dots,x^r_{m_r}))\|,
%\]
%where $\|\cdot\|$ is norm on $\mathbb{R}^r$ monotone in the positive orthant, and $\mathcal{C}_k$ is a cost for which there exists a monotone norm on $\mathbb{R}^{N_k}$, denoted with $|\cdot|^k$, satisfying \eqref{eqn_main} for every $k\in\{1,\dots,r\}$.
% A monotone norm $|\cdot|_\phi$ on $\R^N$ which satisfies \eqref{eqn_main} is 
%$$|(y^1_1,\dots,y^1_{N_1},\dots,y^r_1,\dots, y^r_{N_r})|_\phi=\||(y^1_1,\dots,y^1_{N_1})|^1,\dots,|(y^r_{1},\dots,y^r_{N_r})|^r\|.$$
\item[(7)] \emph{Mailing problem.} For general $m\geq 2$ and $\alpha>0$ consider the following cost
\[
\cost(z_1,\dots,z_m):=\Big(\sum_{i:z_i\geq 0}z_i\Big)^{\alpha}+\Big|\sum_{i:z_i< 0}z_i\Big|^{\alpha},
\]
Observe that this multi-material cost does not satisfy \eqref{supersymm}. A monotone norm $\|\cdot\|$ on $\R^N$ which satisfies \eqref{eqn_main} is clearly
$$\|(x_1,\dots,x_{N})\|=\|x_+\|_{\ell^p}+\|x_-\|_{\ell^p}$$
where $p=\alpha^{-1}$ and $x_+$ (respectively $x_-$) is obtained by $x$ setting all the negative (respectively positive) coordinates of $x$ equal to zero. Such cost is well-suited to give a better description of the discrete \emph{mailing problem} (see \cite{Bernot2009}), encoding the fact that, on every branch of a transportation network, there is a gain in the cost of the transportation in grouping particles flowing with the same orientation, but there should be no gain for two groups of particles flowing with opposite orientations.  
\end{itemize}

{\subsection{Examples of calibrations}
We now focus on elementary multi-material transportation problems with different costs, for which we are able to exhibit constant calibrations. 

\begin{itemize}
\item[(1)]\emph{Mailing problems with Steiner cost.} 
Let us consider the multi-material cost $\cost:\Z^2\to\R$ defined by $\cost(x,y)=|x|^0+|y|^0$, where we mean that $0^0=0$. 

Firstly, we consider the vertices of an isosceles triangle in $\R^2$, for instance $p_1:=(\ell,h), p_2:=(\ell,-h)$ and $p_3:=(0,0)$ for some positive numbers $h,\ell$ with $\ell\geq\sqrt{3}h$, and we fix as a boundary $$\mathcal{B}:=(-1,-1)\delta_{p_3}+ (1,0)\delta_{p_1} +(0,1)\delta_{p_2}.$$ 

A solution to the MMTP associated to such boundary and cost is a 1-dimensional integral $\Z^2$-current supported on a Y-shaped graph, with angles of $2\pi/3$ between the segments at the junction point {(see Figure \ref{h})}. 
\begin{center}
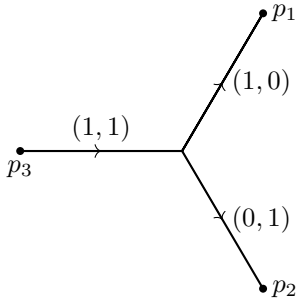
\begin{figure}[h]
\begin{tikzpicture}
   \begin{axis}[axis x line=none,axis y line=none,xmin=-2.4,xmax=2.4,
    ymin=-2,ymax=2,]
      \addplot[thick]coordinates{
(-1.5,0) (0,0) (3/4,3/4*1.71) (0,0)  (3/4,-3/4*1.71)};
\addplot[->]coordinates{
(0,0) (3/8,-3/8*1.71)}
node[right]{$(0,1)$};
\addplot[->]coordinates{
(0,0) (3/8,3/8*1.71)}
node[right]{$(1,0)$};
\addplot[->]coordinates{
(-3/2,0) (-3/4,0) }
node[above]{$(1,1)$};
\addplot[]coordinates{
(-3/2,0)}
node[below]{$p_3$};
\addplot[]coordinates{
(-3/2,0)}
node[circle,fill=black,inner sep=0pt,minimum size=3pt]{};
\addplot[]coordinates{
(3/4,3/4*1.71)}
node[right]{$p_1$};
\addplot[]coordinates{
(3/4,3/4*1.71)}
node[circle,fill=black,inner sep=0pt,minimum size=3pt]{};
\addplot[]coordinates{
(3/4,-3/4*1.71)}
node[right]{$p_2$};
\addplot[]coordinates{
(3/4,-3/4*1.71)}
node[circle,fill=black,inner sep=0pt,minimum size=3pt]{};
   \end{axis} 
\end{tikzpicture}
\caption{{A solution to the mailing problem with Steiner cost for the boundary $\mathcal{B}$.}}
\label{h}
\end{figure}
\end{center}

In order to translate this into a MMP, we endow $\R^2$ with the norm $\|\cdot\|$ which has the unit ball depicted in Figure \ref{Steiner_ball}. 

\begin{center}
\begin{figure}[h!]
\begin{tikzpicture}
   \begin{axis}[axis x line=middle,axis y line=middle,xmin=-2.4,xmax=2.4,ticks = none,
    ymin=-2,ymax=2,]
      \addplot[thick,name path=X,fill = gray,opacity =0.2] coordinates{
(-1,0) (0,1) (1,1) (1,0)  (0,-1) (-1,-1) (-1,0)};
 
   \addplot[name path=X] coordinates{
(-1,0) (0,1) (1,1) (1,0)  (0,-1) (-1,-1) (-1,0)};

  \addplot[] coordinates{
(1,0)}
node[below right] {1};
\addplot[] coordinates{
(-1,0)}
node[below left] {-1};
\addplot[] coordinates{
(0,1)}
node[above right] {1};
\addplot[] coordinates{
(0,-1)}
node[below right] {-1};
   \end{axis} 
\end{tikzpicture}
\caption{Unit ball for the norm $\|\cdot\|$ associated to the cost $\cost$.}
\label{Steiner_ball}
\end{figure}
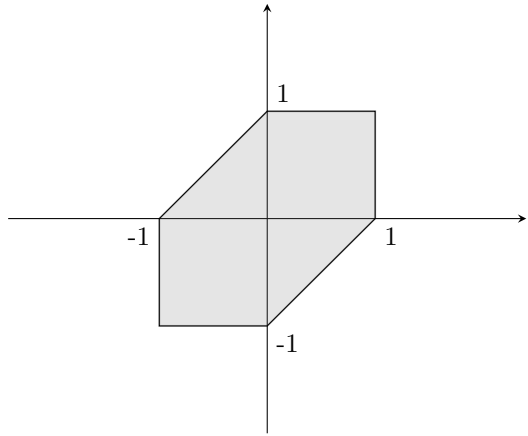
\end{center}

In this special case we have to solve a mass-minimization problem for the same boundary $\mathcal{B}$. Now we show that the $\R^2$-valued differential 1-form represented by the matrix
\begin{equation}\label{calib_z2}
\omega_1:=\left(\begin{array}{lr}
\frac 12 & \frac{\sqrt 3}{2} \\
\frac 12 & -\frac{\sqrt 3}{2}
\end{array}\right)
\end{equation} is a calibration for the minimizer. Indeed $$\langle\omega_1;\left(1/2,\sqrt 3/2\right),(1,0)\rangle=\langle\omega_1;\left(1/2,-\sqrt 3/2\right),(0,1)\rangle=\langle\omega_1;(1,0),(1,1)\rangle=1$$ and the form is constant, hence properties (i) and (ii) in Definition \ref{def:calib} are fulfilled. Moreover, to check (iii), notice that, for every $\phi\in\R$ and every pair of $(g_1,g_2)\in\R^2$ with $\|(g_1,g_2)\|= 1$, we have
$$\left|\langle\omega_1;(\cos\phi,\sin\phi),(g_1,g_2)\rangle\right| =\left|\left(\frac 12\cos\phi+\frac{\sqrt 3}{2}\sin\phi\right)g_1+\left(\frac 12\cos\phi-\frac{\sqrt 3}{2}\sin\phi\right)g_2\right| \le 1\,,$$
where the inequality can be inferred from that fact that the expression in the absolute value is linear in $(g_1,g_2)$ and takes its maximum at some extremal point of the set depicted in Figure \ref{Steiner_ball}, where the values are $\pm\left(\frac 12\cos\phi+\frac{\sqrt 3}{2}\sin\phi\right)=\pm\left(\sin\left(\frac{\pi}{6}+\phi\right)\right), \pm\left(\frac 12\cos\phi-\frac{\sqrt 3}{2}\sin\phi\right)=\pm\left(\sin\left(\frac{\pi}{6}-\phi\right)\right)$, and $\pm\cos\phi$.
%Arguing in the same way for the other orthants, we conclude that $\omega$ is a calibration.

Let us now fix as a boundary (supported on the same points) $$\mathcal{B}':=(1,-1)\delta_{p_3}+ (-1,0)\delta_{p_1} +(0,1)\delta_{p_2}.$$ 
A minimizer in this case is supported in the union of the two segments joining $p_1$ to $p_3$ and $p_3$ to $p_2$, respectively {(see Figure \ref{i})}.

\begin{center} 
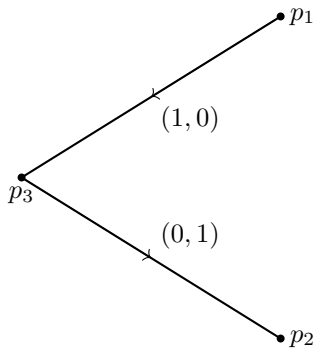
\begin{figure}[h]
\begin{tikzpicture}
   \begin{axis}[axis x line=none,axis y line=none,xmin=-2,xmax=2,
    ymin=-2,ymax=2,]
      \addplot[thick]coordinates{
(1,-1.5)(-1,0) (1,1.5)};
\addplot[->]coordinates{
(-1,0) (0,-0.75)}
node[above right]{$(0,1)$};
\addplot[->]coordinates{
(1,1.5)(0,0.75) }
node[below right]{$(1,0)$};

\addplot[]coordinates{
(-1,0)}
node[below]{$p_3$};
\addplot[]coordinates{
(-1,0)}
node[circle,fill=black,inner sep=0pt,minimum size=3pt]{};
\addplot[]coordinates{
(1,-1.5)}
node[right]{$p_2$};
\addplot[]coordinates{
(1,-1.5)}
node[circle,fill=black,inner sep=0pt,minimum size=3pt]{};
\addplot[]coordinates{
(1,1.5)}
node[right]{$p_1$};
\addplot[]coordinates{
(1,1.5)}
node[circle,fill=black,inner sep=0pt,minimum size=3pt]{};
   \end{axis} 
\end{tikzpicture}
\caption{{{A solution to the mailing problem with Steiner cost for the boundary $\mathcal{B}'$.}}}
\label{i}
\end{figure}
\end{center}

A calibration is the $\R^2$-valued differential $1$-form represented by the matrix 
\[
\omega_2:=\left(\begin{array}{rr}
-\cos \theta & -\sin\theta \\
\cos\theta & -\sin\theta
\end{array}\right)\,,
\] 
where $\theta$ is the positive angle between the segment $\overline{p_3p_1}$ and the horizontal axis. Again, properties (i) and (ii) of Definition \ref{def:calib} are verified by construction of this $\R^2$-valued constant $1$-form. Moreover, to test (iii), we notice that, for every $\phi\in\R$ and every pair $(g_1,g_2)\in\R^2$ with $\|(g_1,g_2)\|= 1$, we have

$$\left|\langle\omega_2;(\cos\phi,\sin\phi),(g_1,g_2)\rangle\right| =\left|-(\cos\theta\cos\phi+\sin\theta\sin\phi)g_1+(\cos\theta\cos\phi-\sin\theta\sin\phi)g_2\right|\le 1\,,$$
again because the expression in the absolute value is linear in $(g_1,g_2)$ and takes its maximum at some extremal point of the set depicted in Figure \ref{Steiner_ball}, where the values are $\pm(\cos\theta\cos\phi+\sin\theta\sin\phi)=\pm\cos(\theta-\phi)$, $\pm(\cos\theta\cos\phi-\sin\theta\sin\phi)=\pm\cos(\theta+\phi)$, and $\pm 2\sin(\theta)\sin(\phi)$ (observe that $\sin(\theta)\leq\frac{1}{2}$ due to the choice of $h$ and $\ell$).

Lastly, let us consider a boundary datum which is supported on the vertices of a square, say $q_1:=(0,1),q_2:=(1,1),q_3:=(1,0),q_4:=(0,0)$. We set $$\mathcal{B}'':= (0,-1)\delta_{q_1}+(1,0)\delta_{q_2}+(0,1)\delta_{q_3}+(-1,0)\delta_{q_4}.$$ A minimizer is given in {Figure \ref{lz}} and it is supported on the set $\Sigma_1$, which is one of the two well-known solutions to the Steiner tree problem for the vertices of the square. This is calibrated by the same $\omega_1$ as in \eqref{calib_z2}, for which all the checks have already been done.
\begin{center}
\begin{figure}[h]
\begin{tikzpicture}
   \begin{axis}[axis x line=none,axis y line=none,xmin=-2.4,xmax=2.4,
    ymin=-2,ymax=2,]
      \addplot[thick]coordinates{
(-5/4,3/4*1.71) (-4/8,0) (-5/4,-3/4*1.71) (-0.5,0) (0.5,0) (5/4,3/4*1.71) (0.5,0)  (5/4,-3/4*1.71)};

\addplot[]coordinates{
(-5/4,3/4*1.71)}
node[circle,fill=black,inner sep=0pt,minimum size=3pt]{};
\addplot[]coordinates{
(-5/4,3/4*1.71)}
node[above]{$q_1$};
\addplot[]coordinates{
(-5/4,-3/4*1.71)}
node[circle,fill=black,inner sep=0pt,minimum size=3pt]{};
\addplot[]coordinates{
(-5/4,-3/4*1.71)}
node[below]{$q_4$};
\addplot[]coordinates{
(5/4,3/4*1.71)}
node[circle,fill=black,inner sep=0pt,minimum size=3pt]{};
\addplot[]coordinates{
(5/4,3/4*1.71)}
node[above]{$q_2$};
\addplot[]coordinates{
(5/4,-3/4*1.71)}
node[circle,fill=black,inner sep=0pt,minimum size=3pt]{};
\addplot[]coordinates{
(5/4,-3/4*1.71)}
node[below]{$q_3$};

\addplot[->]coordinates{
(-5/4,3/4*1.71) (-7/8,3/8*1.71)}
node[left]{$(0,1)$};
\addplot[->]coordinates{
(-5/4,-3/4*1.71) (-7/8,-3/8*1.71)}
node[left]{$(1,0)$};
\addplot[->]coordinates{
(1/2,0) (7/8,3/8*1.71)}
node[right]{$(1,0)$};
\addplot[->]coordinates{
(1/2,0) (7/8,-3/8*1.71)}
node[right]{$(0,1)$};
\addplot[->]coordinates{
(-1/2,0) (0,0)}
node[above]{$(1,1)$};
   \end{axis} 
\end{tikzpicture}
\caption{{A solution to the mailing problem with Steiner cost for the boundary $\mathcal{B}''$.}}
\label{lz}
\end{figure}
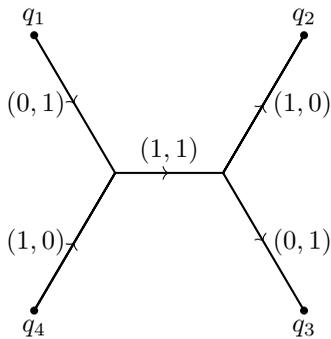
\end{center}   
Observe also that the other solution $\Sigma_2$ to the Steiner tree problem (namely the one obtained from a rotation of $\Sigma_1$ by $90^\circ$) does not support any solution to the MMTP for the boundary $\mathcal{B}''$. Indeed there is only a 1-dimensional integral $\Z^2$-current $T$ with $\partial T=\mathcal{B}''$ supported on $\Sigma_2$, and by direct computation one can see that this is not a minimizer. Note that $T$ has a vertical stretch (oriented by the vector $(0,1)$) carrying the multiplicity $(-1,1)$ and we have $\langle\omega_1;(0,1),(-1,1)\rangle=\sqrt{3}$, hence property (i) of Definition \ref{def:calib} is not satisfied. Since a calibration always calibrates all the minimizers of the problem, this is another proof that $T$ is not a minimizer.  

\item[(2)]\emph{Mailing problems with Gilbert-Steiner cost.} Let us consider the multi-material cost $\cost:\Z^2\to\R$ defined by $\cost(x,y)=\sqrt{|x|^2+|y|^2}$. Now consider the boundary datum 
$$\mathcal{B}:=(N,0)\delta_{r_1}+(0,1)\delta_{r_2}+(-N,-1)\delta_{r_3},$$ where $r_1:=(0,1), r_2:=(1,0),r_3:=(-\cos\tilde\theta,-\sin\tilde\theta)$, with 
$$
\tilde\theta:=\arccos \left(\frac{1}{\sqrt{N^2+1}}\right)\,.
$$ 
Then the Y-shaped graph made by three segments joining at the origin $(0,0)$ supports a solution of the multi-material transport problem {(see Figure \ref{m})}.

\begin{center}
\begin{figure}[h]
\begin{tikzpicture}
   \begin{axis}[axis x line=none,axis y line=none,xmin=-2.4,xmax=2.4,
    ymin=-2,ymax=2,]
      \addplot[thick]coordinates{
(0,1.5) (0,0) (-1/2,-1.4142 ) (0,0)  (1.5,0)};
\addplot[->]coordinates{
(0,0) (0,0.75)}
node[right]{$(N,0)$};
\addplot[->]coordinates{
(0,0) (0.75,0)}
node[above]{$(0,1)$};
\addplot[->]coordinates{
(-1/2,-1.4142 ) (-1/4,-1.4142 /2) }
node[right]{$(N,1)$};
\addplot[]coordinates{
(3/2,0)}
node[below]{$r_2$};
\addplot[]coordinates{
(0,1.5)}
node[circle,fill=black,inner sep=0pt,minimum size=3pt]{};
\addplot[]coordinates{
(0,1.5)}
node[above]{$r_1$};
\addplot[]coordinates{
(-1/2,-1.4142)}
node[circle,fill=black,inner sep=0pt,minimum size=3pt]{};
\addplot[]coordinates{
(-1/2,-1.4142)}
node[right]{$r_3$};
\addplot[]coordinates{
(3/2,0)}
node[circle,fill=black,inner sep=0pt,minimum size=3pt]{};
   \end{axis} 
\end{tikzpicture}
\caption{{A solution to the mailing problem with the Gilbert-Steiner cost for the boundary $\mathcal{B}$.}}
\label{m}
\end{figure}
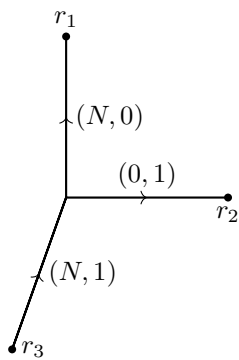
\end{center}

To prove this, firstly we observe that the associated mass-minimization problem for currents with coefficients in $\Z^{N+1}$ has boundary equal to
$$(-1,\dots,-1,-1)\delta_{r_3}+(1,\dots,1,0)\delta_{r_1}+(0,\dots,0,1)\delta_{r_2},$$
and a possible norm $\|\cdot\|$ on $\R^{N+1}$ associated to $\cost$ is 
$$\|(g_1,\dots,g_{N+1})\|=\left|\left(\sum_{j=1}^N|g_j|,g_{N+1}\right)\right|.$$
Next we check that a constant calibration for such MMP is given by the $\R^{N+1}$-valued differential 1-form represented by
\[
\omega:=\left(
\begin{array}{cc}
0 & 1\\
\vdots & \vdots \\
0 & 1\\
1 & 0
\end{array}
\right)\,.
\] 
In fact, (i) holds since
$$\left\langle\omega;\left(\frac{1}{\sqrt {N^2+1}},\frac{N}{\sqrt {N^2+1}}\right),(1\dots,1,1)\right\rangle=\sqrt{N^2+1}=\cost(N,1),$$
$$\langle\omega;(0,1),(1,\dots,1,0)\rangle=N=\cost(N,0)$$
and
$$\langle\omega;(1,0),(0,\dots,0,1)\rangle=1=\cost(0,1).$$
Property (ii) is satisfied, as usual, because $\omega$ is constant. Moreover, for every $\phi\in\R$ and every $(g_1,\ldots,g_{N+1})\in\R^{N+1}$ with $\|(g_1,\ldots,g_{N+1})\|= 1$ one has
$$
|\langle\omega;(\cos\phi,\sin\phi),(g_1,\ldots,g_{N+1})\rangle|\leq\langle(|\cos\phi|,|\sin\phi|);\big(|g_{N+1}|,\sum_{j=1}^N |g_j|\big)\rangle,
$$ and by the definition of $\|\cdot\|$, the right hand side is the scalar product between two vectors of $\R^2$ having unit Euclidean norm, hence it is bounded by 1.
\item[(3)] For the linear combinations of costs discussed in point (3) of the previous subsection, stepping back to the specific case of \cite{Cham_Mer_Fer}, we take $K=2$ and $\alpha_1=0$, $\alpha_2=1$. Let us also assume that $\lambda_1+\lambda_2=1$. Hence the single-material cost is $\cost(z)=\lambda_1|z|^0+\lambda_2|z|$. 

We consider the irrigation problem with source of multiplicity 2 in the point $p_3$ and targets with multiplicity 1 in the points $p_1,p_2$, where $p_1$,$p_2$, and $p_3$ are as in point (1) of this subsection. As we have already observed, a norm on $\R^2$ which turns this single-material transport problem into a MMP is $\lambda_1\|\cdot\|_\infty+\lambda_2\|\cdot\|_1$. Nevertheless, since we are free to choose any monotone norm which coincides with the above on the positive orthant, then we decide to choose the norm $\|\cdot\|$ whose unit ball { is depicted in Figure \ref{n}} (such choice is aimed at reducing the number of extreme points of the unit ball, which makes it easier to estimate the comass norm of the form).
\begin{center}
\begin{figure}[h]
\begin{tikzpicture}
   \begin{axis}[axis x line=middle,axis y line=middle,xmin=-2.4,xmax=2.4,ticks = none,
    ymin=-2,ymax=2,]
      \addplot[thick,name path=X,fill = gray,opacity =0.2] coordinates{
(-1,0) (0,1) (2/3,2/3) (1,0)  (0,-1) (-2/3,-2/3) (-1,0)};
 
   \addplot[name path=X] coordinates{
(-1,0) (0,1) (2/3,2/3) (1,0)  (0,-1) (-2/3,-2/3) (-1,0)};
  
  \addplot[] coordinates{
(1,0)}
node[below right] {1};
\addplot[] coordinates{
(-1,0)}
node[below left] {-1};
\addplot[] coordinates{
(0,1)}
node[above right] {1};
\addplot[] coordinates{
(0,-1)}
node[below right] {-1};
 \end{axis} 
\end{tikzpicture}

\caption{{A norm associated to the cost $\cost$.}}
\label{n}

\end{figure}
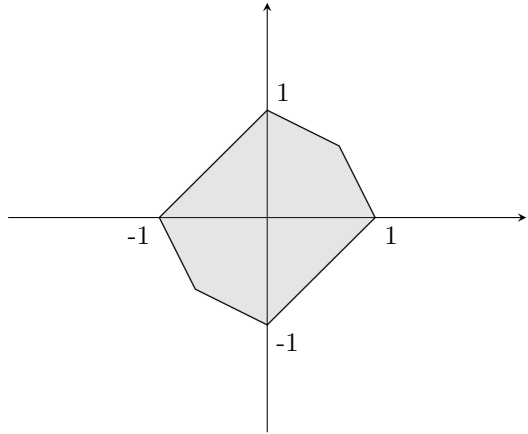
\end{center}

The minimizer for the transportation problem is supported on a Y-shaped graph similar to that shown in point (1), in which the positive angle between the horizontal and the segment joining the branching point to $p_1$ is $\theta=\arccos\left((1+\lambda_2)/2\right)$. A calibration, in this case, is represented by
\[
\omega:=
\left(
\begin{array}{cr}
\frac{1+\lambda_2}{2} & \sqrt{1-\left(\frac{1+\lambda_2}{2}\right)^2} \\
\frac{1+\lambda_2}{2} & -\sqrt{1-\left(\frac{1+\lambda_2}{2}\right)^2}
\end{array}
\right)\,.
\]

To check property (i) of Definition \ref{def:calib} we observe that
$$\langle\omega;(1,0),(1,1)\rangle=1+\lambda_2=\lambda_1+2\lambda_2=\cost(2),$$
$$\langle\omega;(\cos\theta,\sin\theta),(1,0)\rangle=1=\lambda_1+\lambda_2=\cost(1),$$
and
$$\langle\omega;(\cos\theta,-\sin\theta),(0,1)\rangle=1=\lambda_1+\lambda_2=\cost(1).$$
As usual, property (ii) is trivially verified. To check property (iii), we firstly observe that the extreme points of the unit ball for the norm $\|\cdot\|$ are $\pm((1+\lambda_2)^{-1},(1+\lambda_2)^{-1})$, $\pm(1,0)$, and $\pm(0,1)$. Now, for every $\phi\in\R$ we have to check that, whenever $\|(g_1,g_2)\|\leq 1$, it holds
$$\left|\langle\omega;(\cos\phi,\sin\phi),(g_1,g_2)\rangle\right| =\left|(\cos\theta\cos\phi+\sin\theta\sin\phi)g_1+(\cos\theta\cos\phi-\sin\theta\sin\phi)g_2\right|\le 1\,.$$
We observe that the values of the LHS of such inequality at the extreme points above are respectively given by $2(1+\lambda_2)^{-1}|\cos(\theta)\cos(\phi)|=|\cos\phi|$, $|\cos\theta\cos\phi+\sin\theta\sin\phi|=|\cos(\theta-\phi)|$, and $|\cos\theta\cos\phi-\sin\theta\sin\phi|=|\cos(\theta+\phi)|$. Therefore property (iii) is verified.
\end{itemize}

\begin{remark}[Sum of single-material costs]\label{rem:sum}
{\rm To conclude this section, we add a simple, but very useful observation: when the multi-material cost is a composite one (in the sense of point (6) in the previous subsection) but of the form
\[
\cost(z_1,\ldots,z_m)=\cost_1(z_1)+\ldots+\cost_m(z_m)\,,
\]
then the norm $|\cdot|_\star$ in $\R^m$ is the $\ell^1$ norm. Hence roughly speaking, the materials ``do not interact''. More precisely, the minimizer is the sum of the individual minimizers of each (single-material) problem associated to the cost $\cost_i$. This remark matches with the fact that the Monge-type optimal transport of atomic measures is made of ``independent'' segments joining directly the points at the boundary. Moreover, if one can calibrate with $\omega_i$ the problem concerning the $i$-th material with cost $\cost_i$, then a calibration of the global problem is a block-diagonal matrix where each block is given by $\omega_i$.}
\end{remark}
}

\newpage

\bibliographystyle{acm}
\bibliography{Bib_Ann}

\begin{thebibliography}{10}

\bibitem{Bauer_Stoer_Witz}
{\sc Bauer, F.~L., Stoer, J., and Witzgall, C.}
\newblock {Absolute and monotonic norms}.
\newblock {\em Numer. Math. 3\/} (1961), 257--264.

\bibitem{Ber_Ca_Mo}
{\sc Bernot, M., Caselles, V., and Morel, J.-M.}
\newblock {The structure of branched transportation networks}.
\newblock {\em Calc. Var. Partial Differential Equations 32}, 3 (2008),
  279--317.

\bibitem{Bernot2009}
{\sc Bernot, M., Caselles, V., and Morel, J.-M.}
\newblock {\em {Optimal transportation networks}}, vol.~1955 of {\em {Lecture
  Notes in Mathematics}}.
\newblock Springer-Verlag, Berlin, 2009.
\newblock Models and theory.

\bibitem{Bonafini}
{\sc Bonafini, M.}
\newblock Convex relaxation and variational approximation of the {S}teiner
  problem: theory and numerics.
\newblock {\em Geometric Flows 3}, 1 (2018), 19--27.

\bibitem{Bo_Or_Ou}
{\sc Bonafini, M., Orlandi, G., and Oudet, E.}
\newblock {Variational approximation of functionals defined on 1-dimensional
  connected sets: the planar case}.
\newblock {\em SIAM J. Math. Anal. 50}, 6 (2018), 6307--6332.

\bibitem{Bon_Le_San}
{\sc Bonnivard, M., Lemenant, A., and Santambrogio, F.}
\newblock {Approximation of length minimization problems among compact
  connected sets}.
\newblock {\em SIAM J. Math. Anal. 47}, 2 (2015), 1489--1529.

\bibitem{Bran_But_San}
{\sc Brancolini, A., Buttazzo, G., and Santambrogio, F.}
\newblock {Path functionals over {W}asserstein spaces}.
\newblock {\em J. Eur. Math. Soc. (JEMS) 8}, 3 (2006), 415--434.

\bibitem{Bran_So}
{\sc Brancolini, A., and Solimini, S.}
\newblock {Fractal regularity results on optimal irrigation patterns}.
\newblock {\em J. Math. Pures Appl. (9) 102}, 5 (2014), 854--890.

\bibitem{Bran_WirthUP}
{\sc Brancolini, A., and Wirth, B.}
\newblock Equivalent formulations for the branched transport and urban planning
  problems.
\newblock {\em J. Math. Pures Appl. (9) 106}, 4 (2016), 695--724.

\bibitem{Bran_Wirth}
{\sc Brancolini, A., and Wirth, B.}
\newblock General transport problems with branched minimizers as functionals of
  1-currents with prescribed boundary.
\newblock {\em Calculus of Variations and Partial Differential Equations 57}, 3
  (2018).

\bibitem{Bra_But_San}
{\sc Brasco, L., Buttazzo, G., and Santambrogio, F.}
\newblock {A Benamou-Brenier approach to branched transport}.
\newblock {\em SIAM J. Math. Anal. 43}, 2 (2011), 1023--1040.

\bibitem{Ca_Plu}
{\sc Carioni, M., and Pluda, A.}
\newblock Calibrations for minimal networks in a covering space setting.
\newblock Preprint https://arxiv.org/abs/1707.01448, 2017.

\bibitem{Ca_Plu2}
{\sc Carioni, M., and Pluda, A.}
\newblock Calibrations in families for minimal networks.
\newblock {\em Proc. Appl. Math. Mech. 18\/} (2018).

\bibitem{Cham_Mer_Fer}
{\sc Chambolle, A., Ferrari, L. A.~D., and Merlet, B.}
\newblock {A phase-field approximation of the Steiner problem in dimension
  two}.
\newblock Advances in Calculus of Variation. In press. Preprint
  https://arxiv.org/abs/1710.08808.

\bibitem{Co_DeRo_Mar2}
{\sc Colombo, M., {De Rosa}, A., and Marchese, A.}
\newblock {On the well-posedness of branched transportation}.
\newblock Paper in preparation.

\bibitem{Co_DeRo_Mar}
{\sc Colombo, M., De~Rosa, A., and Marchese, A.}
\newblock Improved stability of optimal traffic paths.
\newblock {\em Calc. Var. Partial Differential Equations 57}, 1 (2018).

\bibitem{Co_DeRo_Mar3}
{\sc {Colombo}, M., {De Rosa}, A., and {Marchese}, A.}
\newblock {Stability for the mailing problem}.
\newblock J. Math. Pures Appl. In press. Preprint
  https://arxiv.org/abs/1801.05624, Jan. 2018.

\bibitem{Co_DeRo_Mar_Stu}
{\sc Colombo, M., {De Rosa}, A., Marchese, A., and Stuvard, S.}
\newblock {On the lower semicontinuous envelope of functionals defined on
  polyhedral chains}.
\newblock {\em Nonlinear Anal. 163\/} (2017), 201--215.

\bibitem{Con_Gar_Mas}
{\sc Conti, S., Garroni, A., and Massaccesi, A.}
\newblock Modeling of dislocations and relaxation of functionals on 1-currents
  with discrete multiplicity.
\newblock {\em Calc. Var. Partial Differential Equations 54}, 2 (2015),
  1847--1874.

\bibitem{De_Pauw_Hardt}
{\sc {De Pauw}, T., and Hardt, R.}
\newblock {Rectifiable and flat {$G$} chains in a metric space}.
\newblock {\em Amer. J. Math. 134}, 1 (2012), 1--69.

\bibitem{De_So2}
{\sc Devillanova, G., and Solimini, S.}
\newblock {Elementary properties of optimal irrigation patterns}.
\newblock {\em Calc. Var. Partial Differential Equations 28}, 3 (2007),
  317--349.

\bibitem{De_So}
{\sc Devillanova, G., and Solimini, S.}
\newblock On the dimension of an irrigable measure.
\newblock {\em Rend. Semin. Mat. Univ. Padova 117\/} (2007), 1--49.

\bibitem{Do}
{\sc Dostert, K.}
\newblock {\em {Powerline Communications}}.
\newblock Prentice Hall, 2001.

\bibitem{E_I_Sha}
{\sc Even, S., Itai, A., and Shamir, A.}
\newblock On the complexity of timetable and multicommodity flow problems.
\newblock {\em SIAM J. Comput. 5}, 4 (1976), 691--703.

\bibitem{Fe1}
{\sc Federer, H.}
\newblock {\em {Geometric measure theory}}.
\newblock {Die Grundlehren der mathematischen Wissenschaften, Band 153}.
  Springer-Verlag New York Inc., New York, 1969.

\bibitem{Fer_La_New_Swart}
{\sc Ferreira, H.~C., Lampe, L., Newbury, J., and Swart, T.~G.}
\newblock {\em {Power Line Communications: Theory and Applications for
  Narrowband and Broadband Communications over Power Lines}}.
\newblock John Wiley \& Sons, Ltd, 2010.

\bibitem{Fle}
{\sc Fleming, W.~H.}
\newblock {Flat chains over a finite coefficient group}.
\newblock {\em Trans. Amer. Math. Soc. 121\/} (1966), 160--186.

\bibitem{Ford_Ful}
{\sc Ford, Jr., L.~R., and Fulkerson, D.~R.}
\newblock {\em Flows in networks}.
\newblock Princeton Landmarks in Mathematics. Princeton University Press,
  Princeton, NJ, 2010.
\newblock Paperback edition [of MR0159700], With a new foreword by Robert G.
  Bland and James B. Orlin.

\bibitem{Gold}
{\sc Goldman, M.}
\newblock {Self-similar minimizers of a branched transport functional}.
\newblock Indiana University Mathematics Journal. In press. Preprint
  https://arxiv.org/abs/1704.05342, 2017.

\bibitem{Mad_So_Mo}
{\sc Maddalena, F., Solimini, S., and Morel, J.-M.}
\newblock A variational model of irrigation patterns.
\newblock {\em Interfaces Free Bound. 5}, 4 (2003), 391--415.

\bibitem{AnMa2}
{\sc Marchese, A., and Massaccesi, A.}
\newblock An optimal irrigation network with infinitely many branching points.
\newblock {\em ESAIM Control Optim. Calc. Var. 22}, 2 (2016), 543--561.

\bibitem{AnMa}
{\sc Marchese, A., and Massaccesi, A.}
\newblock The {S}teiner tree problem revisited through rectifiable
  {$G$}-currents.
\newblock {\em Adv. Calc. Var. 9}, 1 (2016), 19--39.

\bibitem{Mar_Mas_Stu_Ti}
{\sc Marchese, A., Massaccesi, A., Stuvard, S., and Tione, R.}
\newblock {A multi-material transport problem with arbitrary marginals}.
\newblock Available at https://arxiv.org/abs/1807.10969, July 2018.

\bibitem{Mas_Ou_Ve}
{\sc Massaccesi, A., Oudet, E., and Velichkov, B.}
\newblock Numerical calibration of {S}teiner trees.
\newblock {\em Applied Mathematics {\&} Optimization 79}, 1 (2019), 69--86.

\bibitem{Mo_San}
{\sc Morel, J.-M., and Santambrogio, F.}
\newblock {The regularity of optimal irrigation patterns}.
\newblock {\em Arch. Ration. Mech. Anal. 195}, 2 (2010), 499--531.

\bibitem{Ou_San}
{\sc Oudet, E., and Santambrogio, F.}
\newblock {A Modica-Mortola approximation for branched transport and
  applications}.
\newblock {\em Arch. Ration. Mech. Anal. 201}, 1 (2011), 115--142.

\bibitem{Pe}
{\sc Pegon, P.}
\newblock On the {L}agrangian branched transport model and the equivalence with
  its {E}ulerian formulation.
\newblock In {\em Topological optimization and optimal transport}, vol.~17 of
  {\em Radon Ser. Comput. Appl. Math.} De Gruyter, Berlin, 2017, pp.~281--303.

\bibitem{Si}
{\sc Simon, L.}
\newblock {\em {Lectures on geometric measure theory}}, vol.~3 of {\em
  {Proceedings of the Centre for Mathematical Analysis, Australian National
  University}}.
\newblock Australian National University Centre for Mathematical Analysis,
  Canberra, 1983.

\bibitem{Smir}
{\sc Smirnov, S.~K.}
\newblock {Decomposition of solenoidal vector charges into elementary
  solenoids, and the structure of normal one-dimensional flows}.
\newblock {\em Algebra i Analiz 5}, 4 (1993), 206--238.

\bibitem{White2}
{\sc White, B.}
\newblock {Rectifiability of flat chains}.
\newblock {\em Ann. of Math. (2) 150}, 1 (1999), 165--184.

\bibitem{White1}
{\sc White, B.}
\newblock {The deformation theorem for flat chains}.
\newblock {\em Acta Math. 183}, 2 (1999), 255--271.

\bibitem{Xia}
{\sc Xia, Q.}
\newblock {Optimal paths related to transport problems}.
\newblock {\em Commun. Contemp. Math. 5\/} (2003), 251--279.

\bibitem{Xia2}
{\sc Xia, Q.}
\newblock {Interior regularity of optimal transport paths}.
\newblock {\em Calc. Var. Partial Differential Equations 20}, 3 (2004),
  283--299.

\bibitem{Xia3}
{\sc Xia, Q.}
\newblock {Boundary regularity of optimal transport paths}.
\newblock {\em Adv. Calc. Var. 4}, 2 (2011), 153--174.

\end{thebibliography}
\end{document}